\documentclass{scrartcl}
\scrollmode
\usepackage{amsmath, amsthm, amssymb,amsfonts,mathtools,bbm,stmaryrd}
\usepackage{verbatim,geometry,graphicx,multirow,array}
\usepackage{algpseudocode}
\usepackage{enumerate}
\usepackage{xcolor}
\usepackage{xy,psfrag}
\usepackage{datetime}

\usepackage{amsmath,amsthm,amssymb,amsfonts,xcolor}
\usepackage{graphicx, caption, subcaption, listings, float}
\usepackage{multirow,enumitem, hyperref, tabularx}
\usepackage{array,algorithm, algpseudocode}
\usepackage{makecell, setspace}
\usepackage{enumitem}
\usepackage{eucal}

\usepackage{xcolor, mathtools}
\usepackage{tabularx,colortbl}

\newcommand{\dd}{\mathrm d}

\newcommand{\R}{{\mathbb R}}

\newcommand{\diag}{\operatorname{diag}}

\renewcommand{\ss}{\scriptstyle}

\def\sddots{\mathinner{\raise3pt\vbox{\hbox{$\ss .$}}
		\raise1.5pt\hbox{$\ss .$}\hbox{$\ss .$}}}

\theoremstyle{plain}
\newtheorem{thm}{Theorem}[section]
\newtheorem{lem}[thm]{Lemma}

\newtheorem{prop}[thm]{Proposition}

\theoremstyle{definition}
\newtheorem{rem}[thm]{Remark}

\usepackage{biblatex}
\newcommand{\icol}[1]{
  \left(\begin{smallmatrix}#1\end{smallmatrix}\right)%
}

\graphicspath{{figures}}

\title{Ordinary differential equations for regularized variational problems involving semi-discrete optimal transport }
\author{
Adrien Cances\thanks{Universit\'e Paris-Saclay, CNRS, Laboratoire de math\'ematiques d'Orsay, ParMA, Inria Saclay, 91405, Orsay, France. 
email: adrien.cances@universite-paris-saclay.fr},\;
Luca Nenna\thanks{Universit\'e Paris-Saclay, CNRS, Laboratoire de math\'ematiques d'Orsay, ParMA, Inria Saclay, 91405, Orsay, France. 
email: luca.nenna@universite-paris-saclay.fr} \thanks{Institut Universitaire de France (IUF)},\;
Daniyar Omarov \thanks{Department of Mathematical and Statistical Sciences, 632 CAB, University of Alberta, Edmonton, Alberta, Canada, T6G 2G1.
email: daniyar@ualberta.ca}
  and  Brendan Pass \thanks{Department of Mathematical and Statistical Sciences, 632 CAB, University of Alberta, Edmonton, Alberta, Canada, T6G 2G1.
email: pass@ualberta.ca}}
\date{\today}

\begin{document}

\maketitle

\begin{abstract}
We consider entropically regularized, semi-discrete versions of variational problems on the set of probability measures involving optimal transport as well as other terms.  We prove that the solutions can be characterized by well-posed ordinary differential equations in the regularization parameter.  The initial conditions for these equations, corresponding to solutions to completely regularized problems, are typically known explicitly.  The ODE can then be solved to recover the solution for an arbitrary degree of regularization; we verify that the solution is continuous in the regularization parameter, implying that taking the limit of the trajectory yields the solution to the fully unregularized problem. We establish analogous results for a version of the problem when the non-optimal transport term is not scaled with the regularization parameter.  We exploit our characterization to numerically solve several example problems using standard ODE methods; this strategy exhibits superior robustness to alternatives such as Newton's method, as arbitrary initializations are not required.
\end{abstract}
\section{Introduction}\label{sec:intro}

Optimal transport (OT) provides a mathematically rigorous framework for comparing and interpolating probability distributions by computing the minimum cost of moving mass from one measure to another.  This simply stated problem enjoys deep connections to a wide variety of mathematical fields, including analysis, geometry, partial differential equations and probability, and a far-reaching suite of applications, in areas such as statistics and machine learning, economics, meteorology and operations research, among others; see, for instance, the monographs  \cite{villani2009optimal,ambrosio2008gradient, Santambrogio2015} for a detailed review.

In this paper, we are interested in class of variational problems involving optimal transport but also other terms.  These problems can be generally formulated as follows:

\begin{equation}\label{eqn: general variational problems}
  \inf_{\substack{\nu\in\mathcal{P}(Y)\\
  \gamma\in\Gamma(\mu,\nu)}}
  \int_{X\times Y} c\;\mathrm{d}\gamma  
  + \,F(\nu)\ ,
\end{equation}
where $\mu \in \mathcal{P}(X)$ is a prescribed probability measure on a set $X$, $\nu$ a free probability measure on a set $Y$, $\Gamma(\mu,\nu) \subset \mathcal{P}(X \times Y)$ the set of couplings of $\mu$ and $\nu$, $c: X \times Y \rightarrow \mathbb{R}$ a given cost function and $F$ a given potential function.  For different forms of $F$ and $c$, this problem models a wide variety of phenomena, including:

\begin{itemize}
    \item When $c(x,y) =\lambda |x-y|^2$ is quadratic on $\mathbb{R}^n \times \mathbb{R}^n$ with $\lambda \in [0,1]$ and $F(\nu)=(1-\lambda)W_2^2(\nu,\rho)$ is $1-\lambda$ times the squared Wasserstein distance between $\nu$ and a second fixed probability measure $\rho \in \mathcal{P}(\mathbb{R}^n)$, minimizers are McCann's displacement interpolants, parametrized by $\lambda$ \cite{McCann1997}. 

    \item Again for $c(x,y) =\frac{1}{2\tau} |x-y|^2$, the minimization represents one step in the discrete time Wasserstein gradient flow of $F$, which, for suitable choices of $F$, are equivalent to various PDE, as established by Jordan, Kinderlehrer, and Otto \cite{jordan1998variational} and systematized in the influential monograph of Ambrosio, Gigli, and Savaré \cite{ambrosio2008gradient}.
 
    \item Cournot-Nash equilibria in game theory, where $c(x,y)$ models preferences of agents $x \in X$ for available strategies $y \in Y$ and the first differential of $F$ represents a cost agents pay for taking the given strategy (which depends on the distribution $\nu$ of strategies the other agents take as well) \cite{BlanchetCarlier2014, BlanchetCarlier2014a, BlanchetCarlier2016}.

    \item In urban planning, where $\mu$ represents a population distribution, $\nu$ a distribution of locations of a given service, $c(x,y)$ the transport cost to move from $x$ to $y$ and $F$ a functional to be minimized (taking into account minimizing congestion, for example) \cite{ButtazzoSantambrogio2005, CarlierSantambrogio2005}.

    \item In hedonic price theory, when $c(x,y)$ represents the cost to a seller $x$ to produce a good $y$, $F(\nu)=\inf_{\pi \in \Pi(\nu,\rho)}\int -b(y,z)d\pi(z,y)$ is the transport cost between $\nu$ and a fixed distribution $\rho$ of buyers with respect to transport cost $-b$, where $b(z,y)$ represents the preference of buyer $z$ for good $y$ \cite{COCV_2005__11_1_57_0}.
\end{itemize}
In practice, when computing solutions, it is often desirable to discretize the state space $Y$, making the measures $\nu$ discrete and the optimal transport problem represented by the first term in \eqref{eqn: general variational problems} semi-discrete.  We note, in addition, that assuming the set $Y$ is discrete is actually natural in some of the applications described above; for instance, in Cournot-Nash problems, there may be only a finite set of possible strategies, while in hedonic pricing problems, there may be only a finite number of different goods which can be traded.  An advantage of semi-discrete optimal transport problems (and, as we will show, of the semi-discrete version of the variational variant \eqref{eqn: general variational problems} as well), is that the problem amounts to a finite dimensional convex optimization problem; see, for instance, the detailed treatment in \cite{MerigotThibert2021}. Efficient Newton-type solvers were developed in \cite{kitagawa2019convergence,levy2015numerical}, and the regularity of Laguerre cells with respect to the dual potential was analyzed in \cite{de2019differentiation}

Another, extremely pervasive,  tool when computing solutions to optimal transport problems is entropic regularization, in which adding an entropic penalty to the discrete transport problem yields a strongly convex objective solvable by Sinkhorn iterations, at a cost several orders of magnitude lower than solving the original linear program \cite{cuturi2013sinkhorn}. In the semi-discrete setting, for the quadratic cost, convergence rates as the regularization parameter vanishes were obtained by Delalande \cite{delalande2022nearly}, who also derived an ODE satisfied by the family of solutions to the semi-discrete entropic problem.  Hiew, Nenna, and Pass \cite{hiew2024ode,nenna2025ode} established that, for fully discrete marginals and general cost functions, the unique solution to the entropically regularized OT problem is in fact \emph{fully characterized} by a well-posed initial-value problem (IVP). Solving this IVP numerically yields a new algorithm whose output interpolates between the fully regularized and unregularized solutions. Nenna, Omarov, and Pass \cite{nenna2025characterizing} proved an analogous result for semi-discrete problems and established a uniform lower bound on the smallest eigenvalue of the Hessian of the dual functional, which is crucial for both theoretical well-posedness and numerical stability.

In this paper, we consider an entropically regularized, semi-discrete version of \eqref{eqn: general variational problems}.  Our main goal, analogous to the work in \cite{hiew2024ode} and \cite{nenna2025characterizing} in the straight optimal transport problem, is to characterize the solution (which now depend on the regularization parameter $1-t$) and exploit it to compute solutions.  As in those works, the initial condition, at $t=0$, can often be readily found in closed form, and, unlike the Sinkhorn algorithm, solving the ODE yields simultaneously the entire curve of solutions.  As we will show, taking the limit as $t \rightarrow 1$ gives the solution to the unregularized problem. 

\subsection*{Contributions}

The main contributions of this paper are as follows.

\begin{enumerate}
  \item \textbf{Uniform boundedness of dual solutions.} We show that the family $\{\psi(t)/t : t \in (0,1)\}$,  where $\psi$ solves the dual \eqref{eq:dual} of our regularized semi-discrete variational problem is uniformly bounded whenever $F^{*}$ is $m$-strongly convex (Lemma \ref{lem:uniformly_bounded_solutions} and Proposition \ref{prop:bounded_nu}). 

  \item \textbf{Well-posedness of the governing ODE (scaled penalty).} By differentiating the first-order optimality condition for problem~\eqref{eq:dual} via the implicit function theorem, we obtain the ODE~\eqref{eq:ODE} satisfied by the trajectory $t \mapsto \psi(t)$. Though this ODE becomes singular at $t=0$, we prove that the corresponding IVP \eqref{eqn: cauchy for z} for $z(t):=\psi(t)/t$  is globally well-posed on $[0,1)$ and so $\psi(t)=tz(t)$, where $z(t)$ is its unique solution, coincides with the unique maximizer of~\eqref{eq:dual} for each $t$ (Theorem~\ref{thm:well-posed0}).

  \item \textbf{Well-posedness for the parameter-independent penalty.} An analogous well-posedness result is established for the variant problem~\eqref{P1:primal} in which $F(\nu)$ is not scaled by $t$, yielding the Cauchy problem~\eqref{cauchy1} (Theorem~\ref{thm:well-posed1}, Section~\ref{sec:noscale}).

  \item \textbf{Continuity of the solution path at $t = 1$.} The map $t \mapsto \psi(t)$ extends continuously to $t = 1$, so that solving the IVP from $t = 0$ to $t = 1$ delivers the dual optimizer of the unregularized problem (Lemma~\ref{lem:limit_psi_at_1}; Lemma~\ref{lem:cont_no_t} for the parameter-independent variant).

  \item \textbf{Computational examples and empirical convergence.} We implement third-order Runge--Kutta solvers (Remark~\ref{rem:RK}) for several example problems in one and two space dimensions, reporting residuals and runtimes in Tables~\ref{tab:prob1_1d}--\ref{tab:prob4_2d}. Empirical first- (and occasionally second-) order convergence in $\Delta t$ is observed.  The ODE-based approach produces the entire solution trajectory $t \mapsto \psi(t)$, enabling visualization of the evolution of Laguerre cells (Figures~\ref{fig:prob3_1d}--\ref{fig:Scal_Par}), and can potentially provide a robust initialization for Newton's method when higher accuracy is desired (Section~\ref{sec:examples}).
\end{enumerate}

\subsection*{Organization}

Section~\ref{sec:prelim} introduces the primal and dual problems, establishes uniform boundedness of dual solutions, and derives the gradient and Hessian formulas for the dual functional $\Phi$. Section~\ref{sec:ode} derives the governing ODE via the implicit function theorem and proves Theorem~\ref{thm:well-posed0} on the $[0,1)$ interval, together with the continuity result at $t = 1$ (Lemma~\ref{lem:limit_psi_at_1}). Section~\ref{sec:noscale} treats the parameter-independent variant, establishing Theorem~\ref{thm:well-posed1}. Section~\ref{sec:examples} presents four computational examples and discusses comparisons with direct Newton solvers. All results confirm that entropic regularization provides a stable and geometrically informative path to the unregularized semi-discrete OT solution, yielding both theoretical guarantees and practical numerical efficiency.

\section{Problem Statement and Preliminary Results}\label{sec:prelim}
Let $Y=\{y_1,\dots,y_N\}\subseteq \R^n$ be a set of $N$ distinct points and let $X\subseteq \R^n$ be a compact  set. 
Let $\mu\in\mathcal{P}_{ac}(X)\cap L^{\infty}(X)$ be an absolutely continuous probability measure on $X$ with an $L^{\infty}$ density; with a slight abuse of notation we also denote its density by $\mu$. The cost function $c:X\times Y\to\R$ is assumed to be $\mathcal{C}^2$  in $x$. We consider the problem
\begin{equation}\label{eq:primal}
    \boxed{\inf_{\substack{\nu\in\mathcal{P}(Y)\\\gamma\in\Gamma(\mu,\nu)}} t\int_{X\times Y} c\;d\gamma + (1-t)\mathrm{Ent}(\gamma|\mu\otimes\sigma) +tF(\nu)\ ,}
\end{equation}
where $t\in[0,1]$, $\sigma=\sum_{i=1}^N\delta_{y_i}$ denotes the counting measure on $Y$, and $F:\R_+^N\to\R\cup\{\infty\}$ is a lower semi-continuous function whose Legendre transform is $\mathcal C^2$ and $m$-strongly convex. We identify the measure $\nu$ with the vector in $\R_+^N$ whose $i$-th component is $\nu_i:=\nu(\{y_i\})$, and $\mathrm{Ent}(\cdot|\mu \otimes \sigma)$ is the Boltzmann-Shannon relative entropy (or Kullback-Leibler divergence) w.r.t.\ the product measure $\mu \otimes \sigma$, defined for general probability measures $p,q$ as
\[
    \mathrm{Ent}(p \,|\, q) = 
    \begin{dcases*}
        \int_{\R^d} \eta \log(\eta)\, \dd q & \text{if $p = \eta q$}\ ,\\
        +\infty & \text{otherwise}\ .
    \end{dcases*}
\]
The fact that $q$ is a probability measure ensures that $\mathrm{Ent}(p \,|\, q) \geq 0$.
 We will also assume that the cost function $c$ is \emph{twisted}, meaning that for each fixed $x \in X$, the mapping 
$$
y_i \mapsto \nabla_xc(x,y_i)
$$
is injective on $Y$.

Introducing Lagrange multipliers $\varphi\in \mathcal{C}_b(X)$ and $\psi\in\R^N$, problem \eqref{eq:primal} can be written as
\begin{align*}
    \inf_{\substack{\nu\in\R^{N}_{ +} \\ \gamma\in\mathcal{M}_+(X\times Y)}} \sup_{\substack{\varphi\in \mathcal{C}_b(X) \\ \psi\in\R^N}}\quad \langle tc, \gamma \rangle + (1-t)\mathrm{Ent}(\gamma|\mu\otimes\sigma) + tF(\nu) + \langle \varphi, \mu\rangle + \langle \psi, \nu\rangle - \langle \varphi\oplus\psi, \gamma\rangle\ .
\end{align*}
The supremum over $\varphi$ enforces that $\gamma$ has first marginal $\mu$, hence $\gamma$ is a probability measure; consequently, the supremum over $\psi$ forces the components $\nu_i$ to be nonnegative and to sum to one. By the Fenchel--Rockafellar theorem we may interchange the supremum and the infimum, yielding
\begin{align*}
    \sup_{\substack{ \varphi\in \mathcal{C}_b(X) \\ \psi\in\R^N}} \inf_{ \substack{\nu\in\R^N_{ +} \\ \gamma\in\mathcal{M}_+(X\times Y)}} \quad \langle tc,\gamma \rangle + (1-t)\mathrm{Ent}(\gamma|\mu\otimes\sigma) + tF(\nu) + \langle \varphi, \mu\rangle + \langle \psi, \nu\rangle - \langle \varphi\oplus\psi, \gamma\rangle\ .
\end{align*}
Identifying the Legendre transforms that arise from the infima over $\nu$ and $\gamma$, and noting that for any fixed $\psi$ the optimal choice of $\varphi$ is
\[
    \varphi(x)=-(1-t)\log\!\Big( \sum_{i=1}^N e^{ \frac{\psi_i-tc(x,y_i)}{1-t}}\Big)\ ,
\]
we obtain the dual problem
\begin{equation} \label{eq:dual}
    \sup_{\psi\in\R^N} \Phi(\psi,t):= \underbrace{-(1-t)\int_X \log\left( \sum_{i=1}^N e^{ \frac{\psi_i-tc(x,y_i)}{1-t} }\right) d\mu(x)}_{=:\mathcal{K}(\psi,t)} -(1-t) -tF^*\left(-\frac{\psi}{t}\right)\ .
\end{equation}

\begin{rem}
    For $t=0$, problem~\eqref{eq:primal} reduces to minimizing the entropy with respect to $\mu\otimes\sigma$. The unique solution is therefore $\nu=\tfrac1N\sigma$ (the uniform probability measure on $Y$) and $\gamma=\mu\otimes\nu$.
\end{rem}

To simplify notation, write $c_i$ for the function $c_i(x)=c(x,y_i)$. Each map $\mathcal{K}(\cdot,t)$ is readily seen to be concave, and elementary computations give the following expressions for its derivatives, which we will use later on:
\begin{equation}\label{eqn: regulairzed grad wrt phi}
    \nabla_\psi\mathcal{K}(\psi,t) = - \int_X \pi^t(\psi)\,d\mu\ ,
\end{equation}
\begin{equation}
    D^2_{\psi,\psi}\mathcal{K}(\psi,t) = \frac{1}{1-t}\int_X [\pi^t(\psi)\pi^t(\psi)^T - \diag(\pi^t(\psi))]\,d\mu\ ,
\end{equation}
\begin{equation}
    [\partial_t\nabla_\psi\mathcal{K}(\psi,t)]_j = \int_X \sum_{k=1,k\neq j}^N \left(\frac{\Delta_{kj}^1(\psi)}{(1-t)^2} \right) \pi_j^t(\psi) \pi_k^t(\psi) \,d\mu\ ,
\end{equation}
where $\Delta^t_{kj}(\psi)= (\psi_k-tc_k)-(\psi_j-tc_j)$ and
\begin{equation}\label{eqn: regularized laguerre integrand}
    [\pi^t(\psi)]_j = \frac{e^{ \frac{\psi_j-tc_j}{1-t}}}{ \sum_{k=1}^N e^{ \frac{\psi_k-tc_k}{1-t}}} = \frac{1}{\sum_{k=1}^N e^{\frac{\Delta_{kj}^t}{1-t}}}\ .
\end{equation}
In particular, $\mathcal{K}(\cdot,t)$ is Lipschitz uniformly in $t$, since each component of its gradient is bounded by one.

\begin{rem}
The calculations above hold for all $t<1$.  The dual of the unregularized problem arising in \eqref{eq:primal} when $t=1$ is
\begin{equation} \label{eq:dual unregularized}
    \sup_{\psi\in\R^N} \Phi(\psi,1):=\underbrace{\sum_{j=1}^N\int_{ \mathrm{Lag}_j(\psi)}[c(x,y_j) -\psi_j]d\mu(x)}_{=:\mathcal{K}(\psi,1)} -F^*\left(-\psi\right)\ ,
\end{equation}
where $\mathrm{Lag}_j(\psi)=\{x\in X : c(x,y_j)-\psi_j \leq c(x,y_k)-\psi_k\ \forall k\in \llbracket 1,N\rrbracket\}$ denotes the Laguerre cell corresponding to $y_j$
 and weights $\psi\in\R^N$.
The components of the gradient of $\mathcal{K}(\psi,1)$ are given by
\begin{equation*}
    \partial_{\psi_j}\mathcal{K}(\psi,1)= -\int_X \mathbbm 1_{\mathrm{Lag}_j(\psi)} \,d\mu = \mu(\mathrm{Lag}_j(\psi))\ .
\end{equation*}
As $t \rightarrow 1^-$, we have that   $\mathcal{K}(\psi,t)\rightarrow \mathcal{K}(\psi,1)$ and $\partial_{\psi_j}\mathcal{K}(\psi,1) \rightarrow    \partial_{\psi_j}\mathcal{K}(\psi,1)$ for each $j$, locally uniformly, by Proposition 53 in \cite{MerigotThibert2021}.
\end{rem}

The derivatives of $\Phi$ which we will need  are
\begin{equation}\label{eqn: gradient of Phi}
    \nabla_\psi\Phi(\psi,t) = \nabla_\psi\mathcal{K}(\psi,t) + \nabla F^*\left(-\frac{\psi}{t}\right)\ ,
\end{equation}
\begin{equation}\label{eqn: Hessian of Phi}
    D^2_{\psi,\psi}\Phi(\psi,t) = D^2_{\psi,\psi}\mathcal{K}(\psi,t) - \frac{1}{t}D^2F^*\left(-\frac{\psi}{t}\right)\ ,
\end{equation}
\begin{equation}\label{eqn: mixed second of Phi}
    \partial_t\nabla_\psi\Phi(\psi,t) = \partial_t\nabla_\psi\mathcal{K}(\psi,t) +\frac{1}{t^2} D^2F^*\left(-\frac{\psi}{t}\right)\psi\ .
\end{equation}

\begin{lem}\label{lem:uniformly_bounded_solutions}
    Let $f : \R^N\to\R$ be differentiable and $m$-strongly convex, and let $g_t:\R^N\to\R$, $t\in(0,1)$, be convex, $\mathcal{C}^2$ differentiable Lipschitz maps with Lipschitz constants uniformly bounded by $L$. For $t\in(0,1)$, let $\psi(t)$ be the unique solution of $\inf_{\psi} g_t(\psi)+tf(\frac{\psi}{t})$. Then $\frac{\psi(t)}{t}$, $t\in(0,1)$, is uniformly bounded.
\end{lem}

\begin{proof}
    Since $f$ is strongly convex, the map $x\mapsto |\nabla f(x)|$ is coercive. Indeed, for any $x\neq0$ we have
    \begin{align*}
    \langle x,\nabla f(x)\rangle = \langle x,\nabla f(0)\rangle + \int_0^1 \underbrace{ \langle x,D^2f(sx)x\rangle}_{\geq m|x|^2}\,ds\ ,
    \end{align*}
    so, dividing by $|x|$ and using the Cauchy--Schwarz inequality, we obtain
    \begin{align*}
    |\nabla f(x)| \geq m|x| - |\nabla f(0)|\ .
    \end{align*}
    The first-order condition reads $\nabla g_t(\psi(t)) + \nabla f\big(\frac{\psi(t)}{t}\big)=0$, hence $|\nabla f(\frac{\psi(t)}{t})|\leq L$. Coercivity of $|\nabla f|$ therefore implies that $\frac{\psi(t)}{t}$ is uniformly bounded for $t\in(0,1)$.
\end{proof}

Thanks to the strong convexity assumption on $F^*$, Lemma~\ref{lem:uniformly_bounded_solutions} directly yields the following result.
\begin{prop}\label{prop:bounded_nu}
    For $t\in(0,1)$, problem~\eqref{eq:dual} has a unique solution $\psi(t)\in\R^N$. Moreover, the family $\{\frac{\psi(t)}{t}: t\in(0,1)\}$ is uniformly bounded; consequently $\psi(t) \xrightarrow[t\to0+]{}0$.
\end{prop}

\section{The Governing ODE}\label{sec:ode}
Since the dual functional $\Phi(\cdot,t)$ is concave it easily follows that the minimizers $\psi(t)$ for any $t$ are characterized by the first order  condition $\nabla_\psi\Phi(\psi(t),t)=0$ which can be differentiated with respect to $t$ via the implicit function theorem, yielding an ODE for the trajectory of solutions. Indeed, the map $(\psi,t)\mapsto\nabla_\psi\Phi(\psi,t)$ is smooth on $(0,1)\times\R^N$, and its Jacobian in $\psi$ is uniformly negative definite because $F^*$ is strongly convex; hence the hypotheses of the implicit function theorem are satisfied. It follows that the trajectory $t \mapsto \psi(t)$ is of class $\mathcal{C}^1$ on $(0,1)$ and satisfies the following ODE on this interval:
\begin{equation} \label{eq:ODE}
    D^2_{\psi,\psi}\Phi(\psi(t),t) \psi'(t) + \partial_t\nabla_\psi\Phi(\psi(t),t) = 0\ .
\end{equation}
Substituting the expressions for the derivatives of $\Phi$, the ODE can be written as
\begin{equation*}
    \left[D^2_{\psi,\psi} \mathcal{K}(\psi(t),t) -\frac{1}{t}D^2F^*\left(-\frac{\psi(t)}{t} \right)\right] \psi'(t) + \partial_t\nabla_\psi\mathcal{K}(\psi(t),t) + \frac{1}{t^2} D^2F^*\left(-\frac{\psi(t)}{t}\right)\psi(t)=0\ .
\end{equation*}

\begin{lem}
    The trajectory $t\mapsto\psi(t)$ is $\mathcal{C}^1$ on $[0,1)$.
\end{lem}
\begin{proof}
    Multiplying the ODE by $t$ and rearranging terms yields
    \begin{align}
        \left[tD^2_{\psi,\psi}\mathcal{K}(\psi(t),t)- D^2F^*\left(-\frac{\psi(t)}{t}\right)\right] &\left(\psi'(t)-\frac{\psi(t)}{t}\right) \nonumber\\[4pt]
        &\quad +\; t\left[D^2_{\psi,\psi}\mathcal{K}(\psi(t),t)\frac{\psi(t)}{t}+\partial_t\nabla_\psi\mathcal{K}(\psi(t),t)\right]= 0\ .\label{eqn: rearranged ODE}
    \end{align}
    By Lemma~\ref{lem:uniformly_bounded_solutions} and the smoothness of $\mathcal{K}$, the expression in the second brackets is bounded, so the second term tends to zero as $t\to 0^+$. Moreover, the concavity of $\mathcal{K}(\cdot,t)$ together with the strong convexity of $F^*$ implies that the matrix in the first brackets is uniformly negative definite for $t$ near $0$. Combining these observations shows that $\psi'(t)-\frac{\psi(t)}{t}$ converges to $0$ as $t\to 0$.
    
    On the other hand, the first-order condition for $t\in(0,1)$ can be written
    \[
        \frac{\psi(t)}{t}=-\nabla F(\nabla_\psi \mathcal{K}(\psi(t),t))\ .
    \]
    Since $\lim_{t\to0}\psi(t)=0$, passing to the limit and using \eqref{eqn: regulairzed grad wrt phi} and \eqref{eqn: regularized laguerre integrand} yields
  \begin{equation}\label{eqn: limit of z}
        \lim_{t\to0}\frac{\psi(t)}{t}= -\nabla F\!\left(\tfrac{1}{N}\mathbf 1\right)\ .
 \end{equation}
    The previous convergence result therefore implies that $t\mapsto \psi(t)$ extends to a $\mathcal{C}^1$ function on $[0,1)$ by setting $\psi(0)=0$, and moreover
\begin{equation}\label{eqn: initial derivative}
    \psi'(0)=-\nabla F\!\left(\tfrac{1}{N}\mathbf 1\right)\ .\qedhere
\end{equation}
\end{proof}

\begin{lem} \label{lem:limit_psi_at_1}
    The map $t\mapsto\psi(t)$ is continuous at $t=1$.
\end{lem}
\begin{proof}
    Let $(t_n)$ be a strictly increasing sequence with $t_n\to 1^-$. By Lemma~\ref{lem:uniformly_bounded_solutions}, the sequence $(\psi(t_n))$ has a convergent subsequence; denote its limit by $\widetilde{\psi}\in\R^N$. The function $\nabla_\psi\mathcal{K}(\cdot,t)$ converges to $\nabla_\psi\mathcal{K}(\cdot,1)$ locally uniformly as $t\to 1$ (see Proposition 53 in \cite{MerigotThibert2021})
    so taking the limit $n\to\infty$ in the first-order condition gives
    \[
        \nabla_\psi\mathcal{K}(\widetilde{\psi},1) + \nabla F^*(\widetilde{\psi})=0\ ,
    \]
    which is precisely the first-order condition at $t=1$. Hence $\widetilde{\psi}=\psi(1)$, and continuity at $t=1$ follows.
\end{proof}

\begin{thm}\label{thm:well-posed0}
    Let $\psi(t)$ be a solution to \eqref{eq:dual} for $t \in [0,1)$ and set $z(t)=\frac{\psi(t)}{t}$. Then the trajectory $t \mapsto z(t)$ is smooth on $(0,1)$ and, for each $t \in (0,1)$, $z(t)$ is the unique solution of the Cauchy problem
    \begin{equation}\label{eqn: cauchy for z}
    \boxed{
    \begin{cases}
        & \left[tD^2_{\psi,\psi}\mathcal{K}(tz(t),t)- D^2F^*\left(-z(t)\right)\right] z'(t) +\left[D^2_{\psi,\psi}\mathcal{K}(tz(t),t)z(t)+\partial_t\nabla_\psi\mathcal{K}(tz(t),t)\right]= 0\ , \\
        &z(0) = -\nabla F(\frac{1}{N}\mathbf 1)\ .
    \end{cases}
    }
    \end{equation}
\end{thm}
\begin{proof}
The limit of $z(t)$ as $t$ tends to $0$ follows from \eqref{eqn: limit of z}.  Dividing \eqref{eqn: rearranged ODE} by $t$ and changing variables from $\psi$ to $z$ then implies that $z$ solves \eqref{eqn: cauchy for z}.  

It remains to show uniqueness; for this, we will apply the Cauchy-Lipschitz theorem to \eqref{eqn: cauchy for z}.  Note that since $z(t)$ has a limit as $t \rightarrow 0$ and  is uniformly bounded in $t$ by Proposition \ref{prop:bounded_nu}, we can work on the compact domain $[0,\bar t] \times \bar B(0,R)$, for some $\bar t <1$ and $R>0$. Since $\mathcal{K}$ and $F^*$ are smooth on this domain, all of their derivatives are bounded there, and by the concavity of $\mathcal{K}(\cdot, t)$ and the strong convexity of $F^*$, we can rewrite the ODE in \eqref{eqn: cauchy for z} as 
\small{
\begin{gather*}
    z'(t) = G(z(t),t) := -\left[tD^2_{\psi,\psi}\mathcal{K}(tz(t),t)- D^2F^*\left(-z(t)\right)\right] ^{-1}\left( D^2_{\psi,\psi}\mathcal{K}(tz(t),t)z(t)+\partial_t\nabla_\psi\mathcal{K}(tz(t),t)\right)\ .
\end{gather*}}
The right hand side is Lipschitz on $[0,\bar t] \times \bar B(0,R)$, and so the Cauchy-Lipschitz theorem then guarantees existence and uniqueness on $[0,\bar t]$. Since $\bar t < 1$ was arbitrary, the result holds on $[0,1)$.
\end{proof}

The preceding result implies that we can recover the solution $\psi(t)=tz(t)$ to the dual problem \eqref{eq:dual} for any $t <1$  by solving the Cauchy problem \eqref{eqn: cauchy for z}.  Lemma \ref{lem:limit_psi_at_1} then implies that we can obtain an arbitrarily good approximation of the solution to the unregularized problem by taking $t$ sufficiently close to $1$.

Note that we prove well-posedness of the Cauchy problem \eqref{eqn: cauchy for z} because the original ODE \eqref{eq:ODE} is not well defined at $t=0$ (as the derivatives in \eqref{eqn: Hessian of Phi} and \eqref{eqn: mixed second of Phi} are not well defined there).  When solving numerically, as an alternative to solving \eqref{eqn: cauchy for z}, we can solve \eqref{eq:ODE} together with the initial condition $\psi(0)=0$, but using directly the derivative $\psi'(0) = -\nabla F(\frac{1}{N}\mathbf 1)$ from \eqref{eqn: initial derivative} in the first time step of the numerical scheme (rather than using the ODE to compute the first value of $\psi'(0)$).  We do this in our computations in Section \ref{sec:examples} below.

\section{Parameter-Independent Functional}\label{sec:noscale}

To motivate this section, let us briefly recap what we have done so far.  The ultimate goal of the ODE method presented above is to characterize and compute solutions to the unregularized problem \eqref{eqn: general variational problems}.

The method to do this developed above is to solve the regularized version \eqref{eq:primal}, or equivalently its dual \eqref{eq:dual} via the ODE \eqref{eqn: cauchy for z} and take the limit as $t \rightarrow 1$.  Since when $t=1$ \eqref{eq:primal} coincides with \eqref{eqn: general variational problems}, this yields the solution to the unregularized problem.

In this section, we present an alternate, but closely related approach; namely, we regularize the problem but do not multiply the functional $F$ by the regularization parameter $t$.  Consider the problem

\begin{equation}\label{P1:primal}
    \inf_{\substack{\nu\in\mathcal{P}(Y)\\\gamma\in\Gamma(\mu,\nu)}} t\int_{X\times Y} c\;d\gamma +(1-t)\mathrm{Ent}(\gamma|\mu\otimes\sigma)+F(\nu)\ , 
\end{equation}
and as above, note that at $t=1$ \eqref{P1:primal} coincides with \eqref{eqn: general variational problems}.  We present here an alternate ODE method to characterize and compute solutions to \eqref{eqn: general variational problems} using this variant; the calculations are very similar, and in several cases simpler than, those for \eqref{eq:primal}. 
For this problem one can derive the dual formulation and its gradient:
\begin{align}\label{P1:dual}
    \sup_{\psi\in\R^N} \Psi(\psi,t):= \mathcal{K}(\psi,t)-(1-t) -F^*\left(-\psi\right)\ , \ 
    \nabla\Psi(\psi,t) = \nabla\mathcal{K}(\psi,t) - \nabla F^*\left(-\psi\right)\ .
\end{align}

\begin{lem}\label{lem:bound_sol}
    Let $f : \R^N \to \R$ be differentiable and $m$-strongly convex, and let $g_t : \R^N \to \R$, $t \in (0,1]$, be convex, $\mathcal{C}^2$ differentiable Lipschitz maps with Lipschitz constants uniformly bounded by $L$. For $t \in [0,1]$, let $\psi(t)$ denote the unique solution of $\inf_{\psi} \, g_t(\psi) + f(\psi)$. Then $\psi(t)$, $t \in [0,1]$, is uniformly bounded.
\end{lem}

\begin{proof}
    The proof follows directly by applying the same arguments as in Lemma \ref{lem:uniformly_bounded_solutions}.
\end{proof}

\begin{lem}\label{lem:cont_no_t}
    The mapping $t\mapsto\psi(t)$ is continuous at $t=1$, where $\psi(t)$ is the unique maximizer of the concave function $\psi \mapsto \Psi(\psi,t)$ in \eqref{P1:dual} at time $t$.
\end{lem}
\begin{proof}
    Let $(t_n)$ be a strictly increasing sequence with $t_n\to 1$. By Lemma \ref{lem:bound_sol}, the family $\{\psi(t_n)\}_n$ is bounded, hence, up to a subsequence, $\psi(t_n)\to\widetilde{\psi}\in\R^N$. The gradient $\nabla_\psi\mathcal{K}(\cdot,t)$ converges to $\nabla_\psi\mathcal{K}(\cdot,1)$ locally uniformly as $t\to 1$ (see Mérigot). Passing to the limit in the first-order optimality condition for $\psi(t_n)$ yields
    \[
        \nabla_\psi\mathcal{K}(\widetilde{\psi},1) + \nabla F^*(\widetilde{\psi}) = 0\ ,
    \]
    which is precisely the first-order condition for the unreqularized dual \eqref{eq:dual unregularized} at $t=1$. Hence $\widetilde{\psi}=\psi(1)$ and the claim follows.
\end{proof}

\begin{thm}\label{thm:well-posed1}
    Let $\psi(t)$ be a solution to \eqref{P1:dual} for $t \in [0,1)$. Then the trajectory $t \mapsto \psi(t)$ is smooth and, for each $t \in [0,1)$, $\psi(t)$ is the unique solution of the Cauchy problem
    \begin{equation}\label{cauchy1}
    \boxed{
    \begin{cases}
        &D_{\psi,\psi}^2\Psi(\mathbf{\psi}(t),t)\mathbf{\psi}'(t) + \dfrac{\partial}{\partial t}\nabla_\psi\Psi(\mathbf{\psi}(t),t) = 0\ ,\quad t \in [0,1)\ , \\
        &\mathbf{\psi}(0)= \log N \cdot \mathbf{1}\ .
    \end{cases}
    }
    \end{equation}
\end{thm}

\begin{proof}
    The proof is the same as that of Theorem \ref{thm:well-posed0}, with Lemma \ref{lem:bound_sol} replacing Lemma \ref{lem:uniformly_bounded_solutions}.
\end{proof}

\section{Computational Examples}\label{sec:examples}

In this section we report numerical simulation results for four variational problems associated with the quadratic cost function \(c(x,y)=\|x-y\|_2^2\), unless otherwise specified. Set
\begin{align}
    \{G_1(\psi, t)\}_j &= e^{-\psi_j} - \int_X \frac{e^{\frac{\psi_j - t c(x, y_j)}{1 - t}}}{\sum_{k=1}^N e^{\frac{\psi_k - t c(x, y_k)}{1 - t}}} \, \dd\mu(x), \quad j = 1, 2, \dots, N\ ,\label{prob1_dual}\\
    \{G_2(\psi,t)\}_j &= e^{-\frac{\psi_j}{t}} - \int_X\frac{e^{\frac{\psi_j - tc(x,y_j)}{1-t}}}{\sum_{k=1}^N e^{\frac{\psi_k - tc(x,y_k)}{1-t}}}\dd\mu(x)\ \ ,\ \ j=1,2,\dots, N\ ,\label{prob2_dual}\\
     \{G_3(\psi, t)\}_j &= e^{-\psi_j} - \int_X \frac{e^{\frac{\psi_j - t c(x, y_j) - v_j}{1 - t}}}{\sum_{k=1}^N e^{\frac{\psi_k - t c(x, y_k) - v_k}{1 - t}}} \, \dd\mu(x), \quad j = 1, 2, \dots, N\ ,\label{prob3_dual}\\
     \{G_4(\psi, t)\}_j &=  \rho[\mathrm{Lag}_j(-\tfrac{\psi_j}{t})] - \int_X \frac{e^{\frac{\psi_j - t c(x, y_j) - v_j}{1 - t}}}{\sum_{k=1}^N e^{\frac{\psi_k - t c(x, y_k) - v_k}{1 - t}}} \, \dd\mu(x), \quad j = 1, 2, \dots, N\ ,\label{prob4_dual}
\end{align}
where \(v_j = \|y_j - P\|^2\) for some point \(P\in X\) and $\rho\in\mathcal{P}_{ac}(X)$ is an absolutely continuous probability measure on $X$. Note that \eqref{prob1_dual} and \eqref{prob2_dual} represent the gradients of the objective functions in the dual problems corresponding to \eqref{P1:primal} and \eqref{eq:primal}, respectively, both associated with
\[
    F(\nu)=\sum_{i=1}^N \nu_i\log\nu_i\ .
\]
The primal problem corresponding to \eqref{prob3_dual} is \eqref{P1:primal} with
\[
    F(\nu)=\sum_{i=1}^N\bigl(\nu_i\log\nu_i+\|y_i-P\|^2\bigr)\ .
\]
The primal problem associated with the gradient of the dual objective function
\eqref{prob4_dual} is
\begin{gather}\label{P4_primal}
    \inf_{\substack{\nu\in\mathcal P(Y)\\\gamma\in\Gamma(\mu,\nu)}} t\int_{X\times Y} c\,\mathrm{d}\gamma +(1-t)\,\mathrm{Ent}(\gamma\mid\mu\otimes\sigma) +tW^2_2(\rho,\nu)\ ,
\end{gather}
where 
\begin{gather*}
    W^2_2(\rho,\nu) = \inf_{\substack{\eta\in\Gamma(\rho,\nu)}} \int_{X\times Y} \|x-y\|_2^2\,\mathrm{d}\eta\ .
\end{gather*}
In the non-regularized case $t=1$, the objective of \eqref{P4_primal} is to determine a probability measure $\nu$ supported on the finite set $Y=\{y_1,\dots,y_N\}$ that balances the transport costs between the two populations. When $c$ is quadratic, this corresponds to a discrete version of the displacement interpolation of McCann \cite{McCann1997}; in another context, it represents the equilibrium distribution of contracts in a hedonic pricing model, where distributions of buyers and sellers are continuous (modeled by $\mu$ and $\rho$ respectively) while the set of possible contracts is discrete, modeled by $Y=\{y_1,...,y_N\}$ \cite{COCV_2005__11_1_57_0}.

For each case we obtain the following initial-value problem (IVP):
\renewcommand\theequation{P\arabic{equation}} \setcounter{equation}{0}
\begin{align}
        \psi'(t) &= -[\nabla G_1(\psi(t), t)]^{-1} \frac{\partial}{\partial t} G_1(\psi(t), t), \quad \psi(0) = \log N \cdot \mathbf{1}\ , 
        \label{prob1_ODE}\\
        \psi'(t) &= - [\nabla G_2(\psi(t),t)]^{-1}\frac{\partial}{\partial t}G_2(\psi(t),t), \ \psi(0) = \mathbf{0}\ , \psi'(0) = -\log{\frac{1}{N}}\times \mathbf{1}\ ,\label{prob2_ODE}\\
        \psi'(t) &= -[\nabla G_3(\psi(t), t)]^{-1} \frac{\partial}{\partial t} G_3(\psi(t), t), \quad \psi(0) = \frac{1}{2}\mathbf{v} + \log \bigg{[}\sum_{k=1}^Ne^{-\frac{1}{2}v_k}\bigg{]}\ ,\label{prob3_ODE}\\
        \psi'(t) &= -[\nabla G_4(\psi(t), t)]^{-1} \frac{\partial}{\partial t} G_4(\psi(t), t), \quad \psi(0) = \mathbf{0}\ , \psi'(0) = -\xi^*\ , \label{prob4_ODE}
\end{align}
where $\xi^*\in\mathbb{R}^N$ is a solution of the problem $\rho[\mathrm{Lag}(\xi)]=\frac{1}{N}\mathbf{1}$. The initial conditions for each problem are obtained either explicitly or from results derived in the previous sections.

In addition, we compute the derivative and Jacobian terms analytically and evaluate the arising integrals numerically using \textbf{MATLAB}'s integration routines: \texttt{integral} for one-dimensional integrals and \texttt{integral2} for two-dimensional integrals.

To solve the IVP-s above we use the Runge--Kutta method from Remark \ref{rem:RK} with parameters \(\alpha=\tfrac{1}{8}\) and \(\beta=\tfrac{1}{4}\). For each problem, we report the sup-norm of the residual at \(t=1\), namely
\[
\texttt{Error} = \|G_k(\psi(1),1)\|_{\infty}\ , \qquad k=1,2,3,4\ .
\]

\begin{rem}[Third-Order Runge-Kutta Method]\label{rem:RK}
   We consider the following two-parameter family of third-order Runge-Kutta methods, parameterized by \(\alpha\) and \(\beta\):
    \begin{gather*}
        k_1 = f(t_n, y_n)\ , \quad
        k_2 = f(t_n + c_2 h, y_n + a_{21} h k_1)\ , \quad
        k_3 = f\!\left(t_n + c_3 h, y_n + h(a_{31} k_1 + a_{32} k_2)\right)\ ,\\
        y_{n+1} = y_n + h \bigl[b_1 k_1 + b_2 k_2 + b_3 k_3\bigr]\ ,
    \end{gather*}
    where \(\alpha,\beta \neq 0\), \(\alpha \neq \beta\), and \(\alpha \neq \tfrac{2}{3}\). The corresponding coefficients are given by
    \begin{gather*}
        a_{21} = \alpha\ , \quad
        a_{31} = \frac{\beta}{\alpha}\frac{\beta - 3\alpha(1-\alpha)}{3\alpha-2}\ , \quad
        a_{32} = -\frac{\beta}{\alpha}\frac{\beta - \alpha}{3\alpha-2}\ ,\\
        b_1 = 1 - \frac{3\alpha + 3\beta - 2}{6\alpha\beta}\ , \quad
        b_2 = \frac{3\beta - 2}{6\alpha(\beta - \alpha)}\ , \quad
        b_3 = \frac{2 - 3\alpha}{6\beta(\beta - \alpha)}\ , \quad
        c_1 = \alpha\ , \quad
        c_2 = \beta\ .
    \end{gather*}
\end{rem}

\subsection{Problems in 1-d}

We perform the following numerical experiment.  For Problems \eqref{prob1_ODE} and \eqref{prob2_ODE} we take the source domain \(X=[0,1]\) and draw the target locations \(y_1,\dots,y_N\) as independent random points in the open interval \((0,5)\).  Problems \eqref{prob3_ODE} and \eqref{prob4_ODE} are treated on the same source domain \(X=[0,1]\) but with target points sampled in the interval \((0,1)\).

For Problems \eqref{prob1_ODE} and \eqref{prob2_ODE} we solve the regularized ODE formulations and, for comparison, we also solve the corresponding unregularized dual problem (which is the same in both cases) using Newton's method.  The Newton implementation and stopping criteria are as described in Remark \ref{rem:Newton1d}.  The comparison highlights differences in convergence behavior and residuals between the ODE-based integrator and the direct Newton solver. All experiments reported below use the same realization of the random targets whenever a direct comparison between methods is made.

\begin{rem}[Newton's Method in 1-d]\label{rem:Newton1d}
    Let \(G(\psi)\) denote the dual functional associated with the unregularized variational problem:
    \begin{gather*}
        \{G(\psi)\}_j = e^{-\psi_j} - \mu[\mathrm{Lag}_j(\psi)], \quad j = 1, 2, \dots, N\ .
    \end{gather*}
    To solve this problem, we apply Newton's method and iterate until convergence:
    \begin{gather}
        \psi^{(k+1)} = \psi^{(k)} - [\nabla G(\psi^{(k)})]^{-1} G(\psi^{(k)})\ .\label{prob5_Newton}
    \end{gather}
    Here,
    \begin{gather*}
        \{G(\psi)\}_j = e^{-\psi_j} - \mu[\mathrm{Lag}_j(\psi)]
        = e^{-\psi_j} - \int_{\mathrm{Lag}_j(\psi)} \mu(x)\,dx\ ,\\
        \{\nabla G(\psi)\}_{ij}
        = \int_{\mathrm{Lag}_i(\psi)\cap \mathrm{Lag}_j(\psi)}
        \frac{\mu(x)}{|\nabla_x ||x-y_i||_2^2 - \nabla_x ||x-y_j||_2^2|}\,ds
        = \frac{\mu(x_{ij})}{2|y_i-y_j|}\ , \ i\neq j\ ,\\
        \{\nabla G(\psi)\}_{ii}
        = -e^{-\psi_i}-\sum_{j=1,j\neq i}^N \{\nabla G(\psi^{(k)})\}_{ij}
        = -e^{-\psi_i}-\sum_{j=1,j\neq i}^N
        \frac{\mu(x_{ij})}{2|y_i-y_j|}\ ,
    \end{gather*}
    \normalsize
    where $i,j = 1,2,\dots, N$ and \(x_{ij} = \mathrm{Lag}_i(\psi)\cap \mathrm{Lag}_j(\psi)\). In all 1-d and 2-d experiments reported below, the Newton iterations were stopped once the following condition was satisfied:
    \[
    \lVert G(\psi^{k}) \rVert_{\infty} < 10^{-8}\ \text{ or }\ k>100
    \]
\end{rem}

Tables \ref{tab:prob1_1d} and \ref{tab:prob2_1d} report the error (sup-norm residual) and CPU time for the IVP-s \eqref{prob1_ODE} and \eqref{prob2_ODE}, respectively. Each experiment was run for the uniform density \(\dd\mu(x)=\dd x\) and the non-uniform density
\(
    \dd\mu(x)=1.8305e^{-10(x-0.5)^2}\dd x .
\)
We observe empirical first-order convergence in the regularization parameter \(t\). In several cases the error stagnates at a fixed mesh resolution, which we attribute to quadrature error: as \(t\to1\) the integrands in the time derivative and Jacobian concentrate (approach a delta), necessitating higher-precision quadrature to retain accuracy.

\begin{table}[ht]
    \centering\scalebox{0.75}{\setstretch{2.5}
    \begin{tabular}{||c||c|c|c|c||}\hline\hline
        $\Delta\mathbf{ t}$ & $N = 2$ & $N = 4$ & $N = 8$ & $N = 16$\\\hline\hline
 
        \multicolumn{5}{c}{\Large \textbf{Uniform Density:} $\dd\mu(x) = \dd x$}\vspace{0.25em}\\\cline{1-5}\cline{1-5}
        $10^{-1}$ & $2.43*10^{-3}$ ($0.17$ sec.) & $6.15*10^{-3}$ ($ 0.08$ sec.) & $1.70*10^{-2}$ ($  0.44$ sec.) & $5.11*10^{-2}$ ($  3.66$ sec.) \\\hline
        $10^{-2}$ & $2.59*10^{-4}$ ($0.27$ sec.) & $2.74*10^{-5}$ ($ 0.95$ sec.) & $1.21*10^{-3}$ ($  9.88$ sec.) & $1.74*10^{-2}$ ($ 48.44$ sec.) \\\hline
        $10^{-3}$ & $1.91*10^{-5}$ ($5.04$ sec.) & $2.59*10^{-5}$ ($28.40$ sec.) & $6.73*10^{-5}$ ($129.46$ sec.) & $1.64*10^{-3}$ ($631.83$ sec.) \\\hline
        \hline

        \multicolumn{5}{c}{\Large \textbf{Non-Uniform Density:} $\dd\mu(x) = 1.8305e^{-10(x-0.5)^2}\dd x$}\vspace{0.25em} \\ \cline{1-5}\cline{1-5}
        $10^{-1}$ & $8.41*10^{-4}$ ($0.59$ sec.) & $2.58*10^{-3}$ ($ 0.09$ sec.) & $2.67*10^{-2}$ ($  0.44$ sec.) & $7.64*10^{-2}$ ($  3.77$ sec.) \\\hline
        $10^{-2}$ & $3.26*10^{-5}$ ($0.46$ sec.) & $1.44*10^{-5}$ ($ 0.79$ sec.) & $4.43*10^{-5}$ ($  9.42$ sec.) & $8.90*10^{-3}$ ($ 45.82$ sec.) \\\hline
        $10^{-3}$ & $6.10*10^{-7}$ ($3.80$ sec.) & $3.97*10^{-6}$ ($14.30$ sec.) & $4.01*10^{-6}$ ($109.92$ sec.) & $6.14*10^{-4}$ ($572.30$ sec.) \\\hline
        \hline
    \end{tabular}}
    \caption{Error and runtime of the \eqref{prob1_ODE} solver for  1-d problems}
    \label{tab:prob1_1d}
\end{table}

\begin{table}[ht]
    \centering\scalebox{0.75}{\setstretch{2.5}
    \begin{tabular}{||c||c|c|c|c||}\hline\hline
        $\Delta\mathbf{ t}$ & $N = 2$ & $N = 4$ & $N = 8$ & $N = 16$\\\hline\hline
 
        \multicolumn{5}{c}{\Large \textbf{Uniform Density:} $\dd\mu(x) = \dd x$}\vspace{0.25em}\\\cline{1-5}\cline{1-5}
        $10^{-1}$ & $1.13*10^{-1}$ ($0.02$ sec.) & $1.21*10^{-1}$ ($ 0.05$ sec.) & $9.41*10^{-2}$ ($ 0.10$ sec.) & $6.28*10^{-2}$ ($  0.41$ sec.) \\\hline
        $10^{-2}$ & $1.06*10^{-2}$ ($0.15$ sec.) & $1.09*10^{-2}$ ($ 0.68$ sec.) & $1.15*10^{-2}$ ($ 2.78$ sec.) & $6.54*10^{-3}$ ($  7.73$ sec.) \\\hline
        $10^{-3}$ & $1.03*10^{-3}$ ($6.57$ sec.) & $1.01*10^{-3}$ ($15.01$ sec.) & $1.17*10^{-3}$ ($36.55$ sec.) & $6.73*10^{-4}$ ($122.44$ sec.) \\\hline
        \hline

        \multicolumn{5}{c}{\Large \textbf{Non-Uniform Density:} $\dd\mu(x) = 1.8305e^{-10(x-0.5)^2}\dd x$}\vspace{0.25em} \\ \cline{1-5}\cline{1-5}
        $10^{-1}$ & $1.13*10^{-1}$ ($0.14$ sec.) & $1.18*10^{-1}$ ($ 0.04$ sec.) & $8.94*10^{-2}$ ($ 0.10$ sec.) & $7.30*10^{-2}$ ($  0.43$ sec.) \\\hline
        $10^{-2}$ & $1.05*10^{-2}$ ($0.27$ sec.) & $1.09*10^{-2}$ ($ 0.62$ sec.) & $1.15*10^{-2}$ ($ 2.89$ sec.) & $7.15*10^{-3}$ ($  6.11$ sec.) \\\hline
        $10^{-3}$ & $1.03*10^{-3}$ ($6.04$ sec.) & $1.09*10^{-3}$ ($12.00$ sec.) & $1.17*10^{-3}$ ($37.80$ sec.) & $6.73*10^{-4}$ ($113.10$ sec.) \\\hline
        \hline
    \end{tabular}}
    \caption{Error and runtime of the \eqref{prob2_ODE} solver for  1-d problems}
    \label{tab:prob2_1d}
\end{table}

Table \ref{tab:prob5_1d} presents the solutions obtained via Newton's method \eqref{prob5_Newton} applied to the unregularized problems, i.e., $t=1$ in \eqref{prob1_dual} and \eqref{prob2_dual}.  Newton's algorithm is faster than the ODE methods, and, when it converges, it achieves substantially smaller residuals. However, its convergence is highly sensitive to the initial guess; for each of the initializations tested here, it failed to reach the target tolerance for \(N=16\).  

\begin{table}[ht]
    \centering\scalebox{0.6}{\setstretch{2.5}
    \begin{tabular}{||c||c|c|c|c||}\hline\hline
        \textbf{Initial Guess} & $N = 2$ & $N = 4$ & $N = 8$ & $N = 16$\\\hline\hline
 
        \multicolumn{5}{c}{\Large \textbf{Uniform Density:} $\dd\mu(x) = \dd x$}\vspace{0.25em}\\\cline{1-5}\cline{1-5}
        $10*\text{rand}(N,1)$  & $0.04         $ ($4.4*10^{-3}$ sec.) & $        0.95$ ($8.1*10^{-3}$ sec.) & $        0.23$ ($1.2*10^{-2}$ sec.) & $        0.10$ ($2.4*10^{-2}$ sec.)\\\hline
        $1*\text{rand}(N,1)$   & $4.14*10^{-13}$ ($3.7*10^{-4}$ sec.) & $2.60*10^{-3}$ ($5.9*10^{-4}$ sec.) & $2.64*10^{-3}$ ($1.4*10^{-2}$ sec.) & $0.05$ ($2.5*10^{-2}$ sec.)\\\hline
        $0.1*\text{rand}(N,1)$ & $1.39*10^{-13}$ ($3.6*10^{-4}$ sec.) & $0.01        $ ($7.8*10^{-3}$ sec.) & $0.02        $ ($1.4*10^{-2}$ sec.) & $0.05$ ($2.9*10^{-2}$ sec.)\\\hline
        $0.01*\text{rand}(N,1)$& $1.06*10^{-12}$ ($1.1*10^{-3}$ sec.) & $0.07        $ ($6.9*10^{-3}$ sec.) & $0.02        $ ($1.3*10^{-2}$ sec.) & $0.06$ ($2.8*10^{-2}$ sec.)\\\hline
        $\mathbf{0}$           & $7.91*10^{-13}$ ($3.6*10^{-4}$ sec.) & $0.07        $ ($6.9*10^{-3}$ sec.) & $0.02        $ ($1.4*10^{-2}$ sec.) & $0.06$ ($2.4*10^{-2}$ sec.)\\\hline
        \hline

        \multicolumn{5}{c}{\Large \textbf{Non-Uniform Density:} $\dd\mu(x) = 1.8305e^{-10(x-0.5)^2}\dd x$}\vspace{0.25em} \\ \cline{1-5}\cline{1-5}
        $10*\text{rand}(N,1)$  & $2.16*10^{-13}$ ($2.2*10^{-3}$ sec.) & $0.03         $ ($9.0*10^{-3}$ sec.) & $4.06*10^{-4}$ ($1.1*10^{-3}$ sec.) & $        0.33$ ($1.8*10^{-2}$ sec.)\\\hline
        $1*\text{rand}(N,1)$   & $1.66*10^{-13}$ ($2.8*10^{-3}$ sec.) & $4.88*10^{-13}$ ($7.2*10^{-4}$ sec.) & $4.24*10^{-4}$ ($1.0*10^{-3}$ sec.) & $8.33*10^{-4}$ ($2.3*10^{-3}$ sec.)\\\hline
        $0.1*\text{rand}(N,1)$ & $2.24*10^{-13}$ ($2.0*10^{-3}$ sec.) & $4.87*10^{-13}$ ($1.5*10^{-3}$ sec.) & $0.01        $ ($1.3*10^{-2}$ sec.) & $        0.10$ ($2.6*10^{-2}$ sec.)\\\hline
        $0.01*\text{rand}(N,1)$& $1.64*10^{-13}$ ($1.9*10^{-3}$ sec.) & $4.78*10^{-13}$ ($3.2*10^{-3}$ sec.) & $0.01        $ ($1.4*10^{-2}$ sec.) & $0.01        $ ($2.8*10^{-2}$ sec.)\\\hline
        $\mathbf{0}$           & $2.16*10^{-13}$ ($1.7*10^{-3}$ sec.) & $5.14*10^{-13}$ ($7.6*10^{-4}$ sec.) & $0.01        $ ($1.4*10^{-2}$ sec.) & $0.01        $ ($2.9*10^{-2}$ sec.)\\\hline
        \hline
    \end{tabular}}
    \caption{Error and runtime of the Newton solver \eqref{prob5_Newton} for  1-d problems} \label{tab:prob5_1d}
\end{table}

Table \ref{tab:prob3_1d} demonstrates results for Problem \eqref{prob3_ODE} with \(P=\tfrac{1}{2}\); here we again observe first-order convergence on average. Figure \ref{fig:prob3_1d} depicts the evolution of the Laguerre cells as $t\to 1$, showing a non-monotone transition as one approaches the unregularized partition.

\begin{table}[ht]
    \centering\scalebox{0.75}{\setstretch{2.5}
    \begin{tabular}{||c||c|c|c|c||}\hline\hline
        $\Delta\mathbf{ t}$ & $N = 2$ & $N = 4$ & $N = 8$ & $N = 16$\\\hline\hline
 
        \multicolumn{5}{c}{\Large \textbf{Uniform Density:} $\dd\mu(x) = \dd x$}\vspace{0.25em}\\\cline{1-5}\cline{1-5}
        $10^{-1}$ & $1.72*10^{- 3}$ ($ 0.08$ sec.) & $2.75*10^{-2}$ ($ 0.05$ sec.) & $7.04*10^{-2}$ ($ 0.12$ sec.) & $3.76*10^{-2}$ ($  0.489$ sec.) \\\hline
        $10^{-2}$ & $7.85*10^{- 7}$ ($ 1.33$ sec.) & $2.75*10^{-3}$ ($ 1.97$ sec.) & $4.29*10^{-3}$ ($ 2.72$ sec.) & $1.52*10^{-2}$ ($ 10.03$ sec.) \\\hline
        $10^{-3}$ & $3.65*10^{-10}$ ($13.26$ sec.) & $2.85*10^{-8}$ ($27.20$ sec.) & $2.53*10^{-4}$ ($67.84$ sec.) & $1.70*10^{-3}$ ($229.72$ sec.) \\\hline
        \hline

        \multicolumn{5}{c}{\Large \textbf{Non-Uniform Density:} $\dd\mu(x) = 1.8305e^{-10(x-0.5)^2}\dd x$}\vspace{0.25em} \\ \cline{1-5}\cline{1-5}
        $10^{-1}$ & $4.56*10^{-3}$ ($ 0.14$ sec.) & $4.02*10^{-2}$ ($ 0.05$ sec.) & $7.03*10^{-2}$ ($ 0.12$ sec.) & $5.15*10^{-2}$ ($  0.47$ sec.) \\\hline
        $10^{-2}$ & $1.42*10^{-6}$ ($ 1.27$ sec.) & $1.15*10^{-3}$ ($ 0.90$ sec.) & $2.10*10^{-2}$ ($ 2.84$ sec.) & $2.05*10^{-2}$ ($ 10.14$ sec.) \\\hline
        $10^{-3}$ & $1.31*10^{-9}$ ($13.33$ sec.) & $1.86*10^{-6}$ ($24.30$ sec.) & $3.82*10^{-3}$ ($60.39$ sec.) & $3.99*10^{-3}$ ($230.44$ sec.) \\\hline
        \hline
    \end{tabular}}
    \caption{Error and runtime of the \eqref{prob3_ODE} solver with for  1-d problems with $P=0.5$}\label{tab:prob3_1d}
\end{table}

\begin{figure}[ht]
    \centering
    \begin{subfigure}{.495\textwidth}
        \centering
        \includegraphics[width=\linewidth]{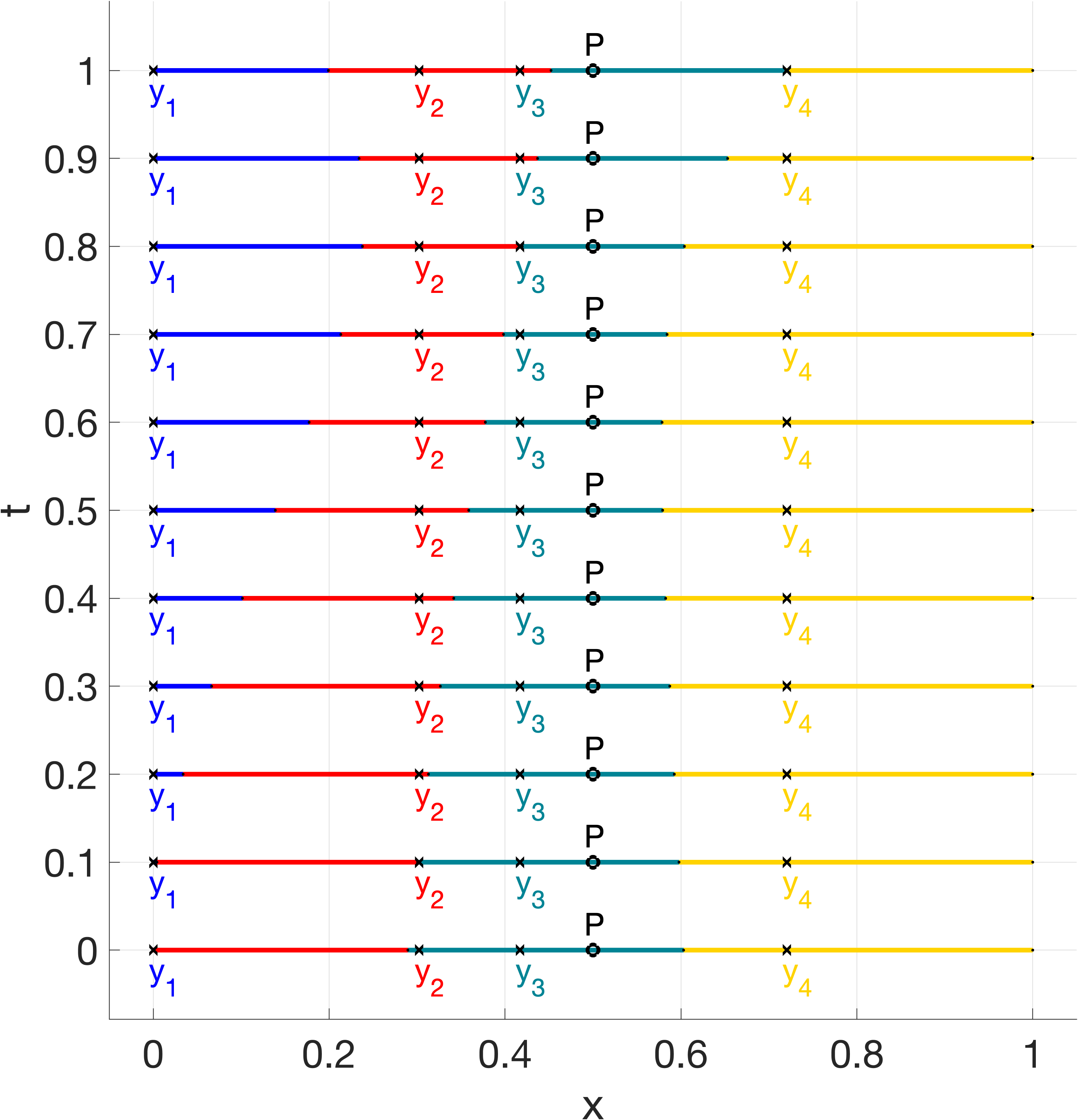}
        \caption{Uniform density}
    \end{subfigure}
    \begin{subfigure}{.495\textwidth}
        \centering
        \includegraphics[width=\linewidth]{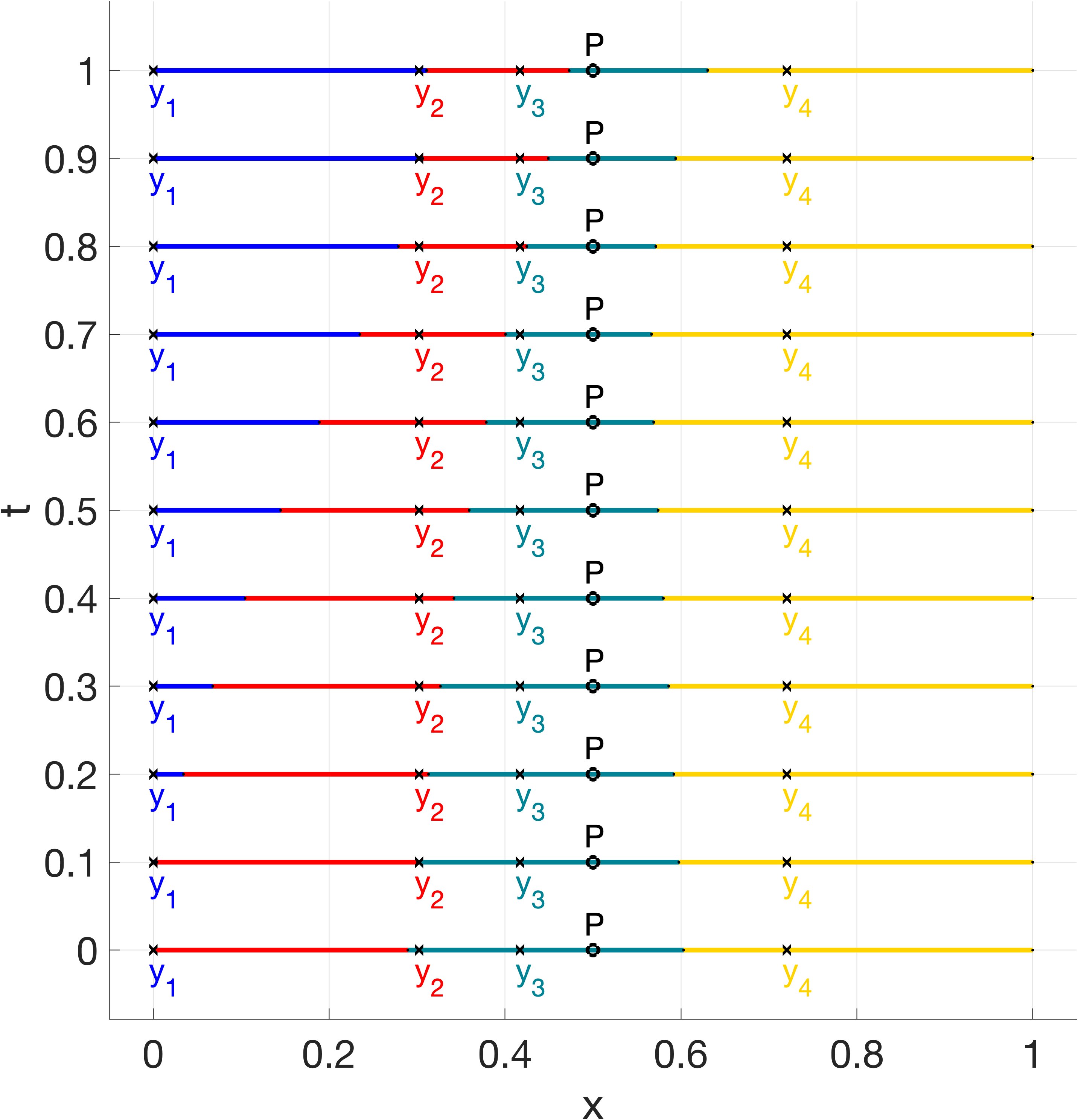}
        \caption{Non-uniform density}
    \end{subfigure}
    \caption{Time evolution of Laguerre cells for Problem \eqref{prob3_ODE} with $P = \tfrac12$ and $N = 4$}
    \label{fig:prob3_1d}
\end{figure}

Table \ref{tab:prob4_1d} reports the error and runtime for the solver \eqref{prob4_ODE} with cost function \(c(x,y)=\|x-y\|_2^3\) and measures \(\mathrm{d}\mu(x)=\mathrm{d}x\) and \(\mathrm{d}\rho(x)=1.8305 \exp\big(-10(x-0.5)^2\big)\,\mathrm{d}x\). The results demonstrate convergence with respect to $t$, with the computational time increasing notably as the temporal discretization is refined.

\begin{table}[ht]
    \centering\scalebox{0.75}{\setstretch{2.5}
    \begin{tabular}{||c||c|c|c|c||}\hline\hline
        $\Delta\mathbf{ t}$ & $N = 3$ & $N = 6$ & $N = 12$ & $N = 24$\\\hline\hline
        $10^{-1}$ & $1.83*10^{-2}$ ($ 0.16$ sec.) & $9.56*10^{-2}$ ($ 0.56$ sec.) & $1.22*10^{-1}$ ($  3.96$ sec.) & $9.13*10^{-2}$ ($  27.85$ sec.) \\\hline
        $10^{-2}$ & $3.14*10^{-3}$ ($ 1.00$ sec.) & $4.51*10^{-2}$ ($ 5.22$ sec.) & $4.88*10^{-2}$ ($ 41.46$ sec.) & $8.31*10^{-2}$ ($ 301.33$ sec.) \\\hline
        $10^{-3}$ & $8.20*10^{-4}$ ($10.50$ sec.) & $2.99*10^{-2}$ ($52.02$ sec.) & $2.99*10^{-2}$ ($415.49$ sec.) & $4.09*10^{-2}$ ($3099.79$ sec.) \\\hline
        \hline
    \end{tabular}}
    \caption{Error and time for \eqref{prob4_ODE} with $\dd\mu(x) = \dd x$ and $\dd\rho(x) = 1.8305e^{-10(x-0.5)^2}\dd x$}\label{tab:prob4_1d}
\end{table}

\subsection{Problems in 2-d}

We now consider the two-dimensional setting. For Problems \eqref{prob1_ODE} and \eqref{prob2_ODE} the source domain is \(X=[0,1]\times[0,1]\) and the target locations \(y_1,\dots,y_N\) are sampled in $\left[0,\tfrac{3}{2}\right]\times\left[0,\tfrac{3}{2}\right]$. Problems \eqref{prob3_ODE} and \eqref{prob4_ODE} are posed on the same source domain with target points sampled in \(X\). 

For comparison, the unregularized duals of Problems \eqref{prob1_ODE} and \eqref{prob2_ODE} are also solved by Newton's method. With the quadratic cost \(c(x,y)=\|x-y\|_2^2\), Laguerre cells are computed efficiently via the lifting algorithm (see \cite{levy2015numerical}), and the Jacobian is assembled from the analytic expressions in \cite{kitagawa2019convergence,de2019differentiation}. 

\renewcommand\theequation{\arabic{equation}} \setcounter{equation}{24}
Tables \ref{tab:prob1_2d} and \ref{tab:prob2_2d} report the errors and computation times for the IVP-s \eqref{prob1_ODE} and \eqref{prob2_ODE}, respectively. Each experiment was run for the uniform density \(\dd\mu(x)=\dd x_1\dd x_2\) and for the non-uniform density
\begin{gather}\label{NonUnifDens}
    \dd\mu(x)=3.3508\exp\bigl(-10(x_1-0.5)^2-10(x_2-0.5)^2\bigr)\,\dd x_1\dd x_2\ .
\end{gather}
As in one dimension, we observe empirical first-order convergence with respect to the regularization parameter \(t\). Stagnation of the error at \(\Delta t=10^{-3}\) in several instances is attributable to integration error: as \(t\to1\) the integrands appearing in the time derivative and in the Jacobian approach delta functions, requiring finer quadrature and making accurate two-dimensional integration increasingly critical. Figure \ref{fig:prob1_LagCells} illustrates the evolution of the Laguerre cells as a function of the regularization parameter $t$ under the nonuniform density \eqref{NonUnifDens} for Problem \eqref{prob1_ODE}. Figure \ref{fig:MeasLagCells} shows the corresponding evolution of the measures, while Figure \ref{fig:RLagCells} depicts the evolution of the smoothed Laguerre cell associated with a target point $y_i$,
\begin{gather*}
    \mathrm{RLag}^t_i(\psi) := \dfrac{\exp\!\left(\frac{\psi_i - t\,c(x,y_i)}{1 - t}\right)}{\sum_{k=1}^N \exp\!\left(\frac{\psi_k - t\,c(x,y_k)}{1 - t}\right)} \ .
\end{gather*}

\begin{table}[ht]
    \centering\scalebox{0.75}{\setstretch{2.5}
    \begin{tabular}{||c||c|c|c|c||}\hline\hline
        $\Delta\mathbf{ t}$ & $N = 2$ & $N = 4$ & $N = 8$ & $N = 16$\\\hline\hline
 
        \multicolumn{5}{c}{\Large \textbf{Uniform Density:} $\dd\mu(x) = \dd x$}\vspace{0.25em}\\\cline{1-5}\cline{1-5}
        $10^{-1}$ & $9.99*10^{-4}$ ($ 0.51$ sec.) & $4.01*10^{-3}$ ($  0.63$ sec.) & $1.96*10^{-2}$ ($  2.19$ sec.) & $9.04*10^{-3}$ ($   8.91$ sec.) \\\hline
        $10^{-2}$ & $2.58*10^{-6}$ ($ 1.88$ sec.) & $2.30*10^{-7}$ ($ 14.40$ sec.) & $2.56*10^{-4}$ ($ 39.80$ sec.) & $1.54*10^{-4}$ ($ 134.23$ sec.) \\\hline
        $10^{-3}$ & $2.08*10^{-8}$ ($22.48$ sec.) & $9.17*10^{-4}$ ($168.17$ sec.) & $5.53*10^{-4}$ ($472.14$ sec.) & $2.44*10^{-5}$ ($1547.47$ sec.) \\\hline
        \hline

        \multicolumn{5}{c}{\Large \textbf{Non-Uniform Density:} $\dd\mu(x) = 3.3508e^{-10(x_1-0.5)^2-10(x_2-0.5)^2}\dd x_1\dd x_2$}\vspace{0.25em} \\ \cline{1-5}\cline{1-5}
        $10^{-1}$ & $2.12*10^{-3}$ ($ 0.20$ sec.) & $4.19*10^{-3}$ ($  1.07$ sec.) & $1.86*10^{-2}$ ($  4.52$ sec.) & $1.19*10^{-2}$ ($  20.17$ sec.) \\\hline
        $10^{-2}$ & $8.50*10^{-8}$ ($ 2.43$ sec.) & $1.14*10^{-6}$ ($ 19.10$ sec.) & $2.92*10^{-4}$ ($ 64.44$ sec.) & $3.40*10^{-4}$ ($ 250.97$ sec.) \\\hline
        $10^{-3}$ & $4.68*10^{-7}$ ($33.55$ sec.) & $8.75*10^{-4}$ ($212.03$ sec.) & $8.20*10^{-4}$ ($717.40$ sec.) & $8.65*10^{-4}$ ($2686.31$ sec.) \\\hline
        \hline
    \end{tabular}}
    \caption{Error and runtime of the \eqref{prob1_ODE} solver for  2-d problems}
    \label{tab:prob1_2d}
\end{table}

\begin{table}[ht]
    \centering\scalebox{0.75}{\setstretch{2.5}
    \begin{tabular}{||c||c|c|c|c||}\hline\hline
        $\Delta\mathbf{ t}$ & $N = 2$ & $N = 4$ & $N = 8$ & $N = 16$\\\hline\hline
 
        \multicolumn{5}{c}{\Large \textbf{Uniform Density:} $\dd\mu(x) = \dd x$}\vspace{0.25em}\\\cline{1-5}\cline{1-5}
        $10^{-1}$ & $6.69*10^{-3}$ ($ 0.51$ sec.) & $1.71*10^{-2}$ ($  0.72$ sec.) & $2.11*10^{-2}$ ($  2.25$ sec.) & $1.08*10^{-2}$ ($   8.90$ sec.) \\\hline
        $10^{-2}$ & $6.93*10^{-4}$ ($ 2.13$ sec.) & $1.30*10^{-3}$ ($ 15.39$ sec.) & $6.83*10^{-4}$ ($ 41.12$ sec.) & $4.65*10^{-4}$ ($ 134.85$ sec.) \\\hline
        $10^{-3}$ & $6.87*10^{-5}$ ($24.40$ sec.) & $3.75*10^{-4}$ ($174.99$ sec.) & $5.68*10^{-4}$ ($478.46$ sec.) & $5.00*10^{-5}$ ($1566.33$ sec.) \\\hline
        \hline

        \multicolumn{5}{c}{\Large \textbf{Non-Uniform Density:} $\dd\mu(x) = 3.3508e^{-10(x_1-0.5)^2-10(x_2-0.5)^2}\dd x_1\dd x_2$} \vspace{0.25em} \\ \cline{1-5}\cline{1-5}
        $10^{-1}$ & $5.61*10^{-3}$ ($ 0.31$ sec.) & $1.20*10^{-2}$ ($  1.06$ sec.) & $2.07*10^{-2}$ ($  4.48$ sec.) & $1.36*10^{-2}$ ($  19.62$ sec.) \\\hline
        $10^{-2}$ & $6.96*10^{-4}$ ($ 2.58$ sec.) & $1.30*10^{-3}$ ($ 20.05$ sec.) & $7.03*10^{-4}$ ($ 67.45$ sec.) & $5.16*10^{-4}$ ($ 258.25$ sec.) \\\hline
        $10^{-3}$ & $6.89*10^{-5}$ ($36.01$ sec.) & $1.36*10^{-3}$ ($224.61$ sec.) & $8.59*10^{-4}$ ($746.58$ sec.) & $8.81*10^{-4}$ ($2772.25$ sec.) \\\hline
        \hline
    \end{tabular}}
    \caption{Error and runtime of the \eqref{prob2_ODE} solver for  2-d problems}
    \label{tab:prob2_2d}
\end{table}

\begin{figure}[ht]
    \centering
    \begin{subfigure}{.4\textwidth}
        \centering
        \includegraphics[width=0.8\linewidth]{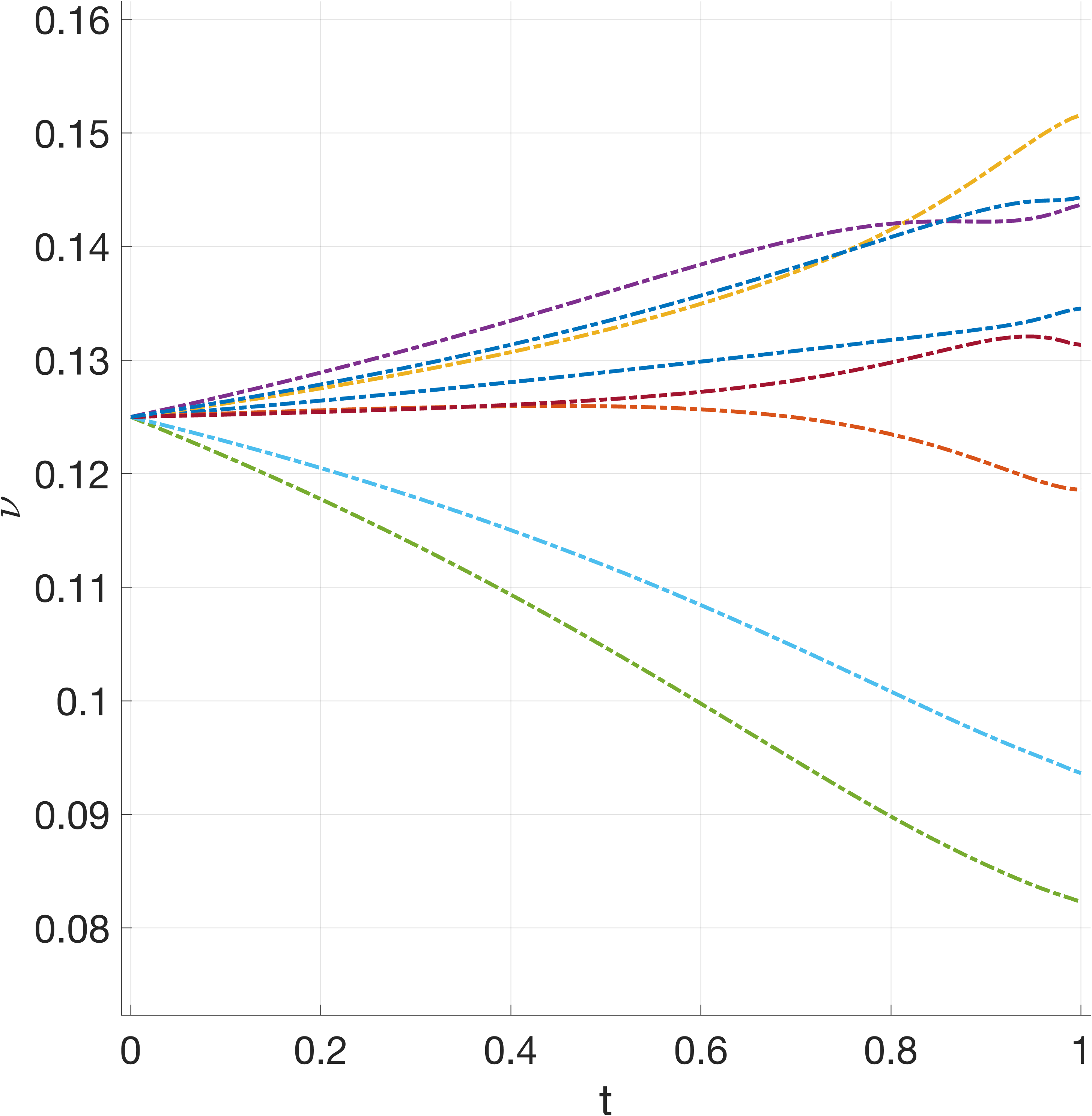}
        \caption{Measure of Laguerre cells, $N=8$}
        \label{fig:MeasLagCells}
    \end{subfigure}
    \begin{subfigure}{.58\textwidth}
        \centering
        \includegraphics[width=0.8\linewidth]{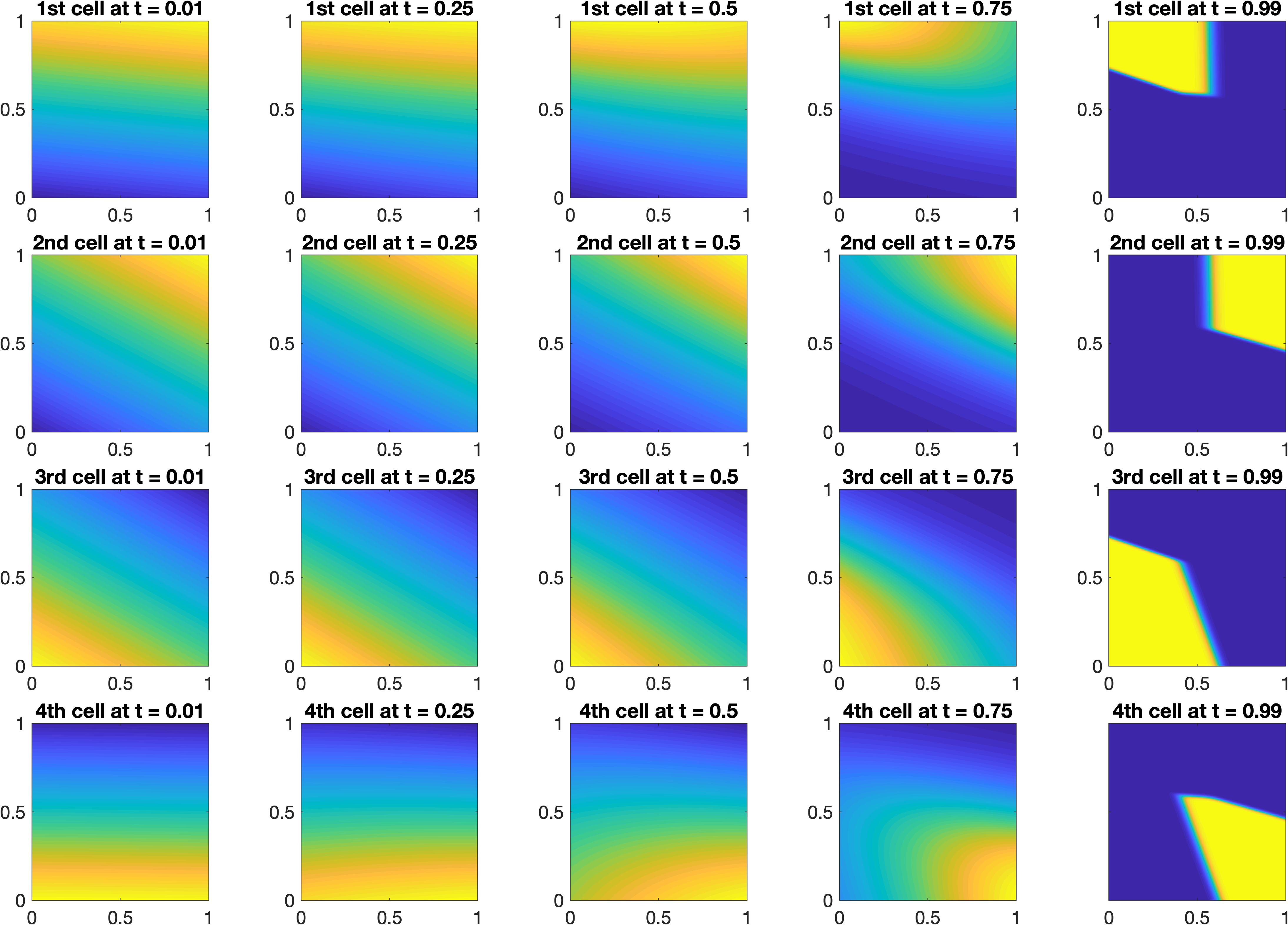}
        \caption{Smoothed Laguerre cells, $N=4$}
        \label{fig:RLagCells}
    \end{subfigure}
    \caption{Time evolution of Laguerre cells for Problem \eqref{prob1_ODE} with nonuniform density \eqref{NonUnifDens}}
    \label{fig:prob1_LagCells}
\end{figure}

Table \ref{tab:prob5_2d} presents results for the Newton solver \eqref{prob5_Newton}. These results may be compared with those in Tables \ref{tab:prob1_2d} and \ref{tab:prob2_2d}, which correspond to the same benchmark problems. When Newton converges it attains smaller residuals, but in two dimensions its convergence is noticeably more sensitive to the initial guess; with the present initialization Newton's method failed to reach the prescribed tolerance for \(N=8\) under the non-uniform measure.  A possible practical strategy could be combining both approaches  to yield a desirable combination of robustness with improved accuracy,  by  using the ODE solution to initialize Newton.

\begin{table}[ht]
    \centering\scalebox{0.75}{\setstretch{2.5}
    \begin{tabular}{||c||c|c|c|c||}\hline\hline
        \textbf{Initial Guess} & $N = 2$ & $N = 4$ & $N = 8$ & $N = 16$\\\hline\hline
 
        \multicolumn{5}{c}{\Large \textbf{Uniform Density:} $\dd\mu(x) = \dd x$}\vspace{0.25em}\\\cline{1-5}\cline{1-5}
        $10*\text{rand}(N,1)$  &                           NAN &                           NAN &                            NAN & NAN\\\hline
        $1*\text{rand}(N,1)$   & $5.51*10^{-10}$ ($0.39$ sec.) & $2.08*10^{-12}$ ($2.40$ sec.) &                            NAN & NAN\\\hline
        $0.1*\text{rand}(N,1)$ & $2.61*10^{- 9}$ ($0.45$ sec.) & $9.41*10^{-10}$ ($2.44$ sec.) & $1.45*10^{-11}$ ($14.11$ sec.) & NAN\\\hline
        $0.01*\text{rand}(N,1)$& $8.05*10^{- 9}$ ($0.46$ sec.) & $1.71*10^{- 9}$ ($2.41$ sec.) & $2.21*10^{-11}$ ($15.36$ sec.) & NAN\\\hline
        $\mathbf{0}$           & $9.07*10^{- 9}$ ($0.45$ sec.) & $1.83*10^{- 9}$ ($2.37$ sec.) &                            NAN & NAN\\\hline
        \hline

        \multicolumn{5}{c}{\Large \textbf{Non-Uniform Density:} $\dd\mu(x) = 3.3508e^{-10(x_1-0.5)^2-10(x_2-0.5)^2}\dd x_1\dd x_2$}\vspace{0.25em} \\ \cline{1-5}\cline{1-5}
        $10*\text{rand}(N,1)$  &                           NAN &                          NAN & NAN & NAN \\\hline
        $1*\text{rand}(N,1)$   & $5.01*10^{-10}$ ($0.42$ sec.) & $4.49*10^{-9}$ ($2.97$ sec.) & NAN & NAN \\\hline
        $0.1*\text{rand}(N,1)$ & $4.69*10^{- 9}$ ($0.49$ sec.) & $1.21*10^{-9}$ ($2.53$ sec.) & NAN & NAN \\\hline
        $0.01*\text{rand}(N,1)$& $3.33*10^{-16}$ ($0.61$ sec.) & $2.33*10^{-9}$ ($2.57$ sec.) & NAN & NAN \\\hline
        $\mathbf{0}$           & $6.66*10^{-16}$ ($0.65$ sec.) & $2.53*10^{-9}$ ($2.55$ sec.) & NAN & NAN \\\hline
        \hline
    \end{tabular}}
    \caption{Error and runtime of the Newton solver \eqref{prob5_Newton} for  2-d problems} \label{tab:prob5_2d}
\end{table}

Table \ref{tab:prob3_2d} illustrates results for Problem \eqref{prob3_ODE} with \(\mathbf{P}=(0.5,0.5)^{\operatorname{T}}\); here we observe on average the second order convergence behavior. Figure \ref{fig:prob3_2d_unif} illustrates the evolution of the Laguerre cells as \(t\to1\) for uniform ($\dd\mu(x) = \dd x$) measure.

\begin{table}[ht]
    \centering\scalebox{0.75}{\setstretch{2.5}
    \begin{tabular}{||c||c|c|c|c||}\hline\hline
        $\Delta\mathbf{ t}$ & $N = 2$ & $N = 4$ & $N = 8$ & $N = 16$\\\hline\hline
 
        \multicolumn{5}{c}{\Large \textbf{Uniform Density:} $\dd\mu(x) = \dd x$}\vspace{0.25em}\\\cline{1-5}\cline{1-5}
        $10^{-1}$ & $5.10*10^{-5}$ ($ 0.08$ sec.) & $4.45*10^{-2}$ ($ 0.28$ sec.) & $2.83*10^{-2}$ ($  1.76$ sec.) & $1.03*10^{-2}$ ($   8.74$ sec.) \\\hline
        $10^{-2}$ & $4.64*10^{-7}$ ($ 2.03$ sec.) & $6.35*10^{-5}$ ($ 7.16$ sec.) & $1.47*10^{-4}$ ($ 32.17$ sec.) & $4.11*10^{-4}$ ($ 113.56$ sec.) \\\hline
        $10^{-3}$ & $1.07*10^{-8}$ ($27.45$ sec.) & $7.74*10^{-6}$ ($94.40$ sec.) & $6.22*10^{-7}$ ($382.60$ sec.) & $4.41*10^{-5}$ ($1308.95$ sec.) \\\hline
        \hline

        \multicolumn{5}{c}{\Large \textbf{Non-Uniform Density:} $\dd\mu(x) = 3.3508e^{-10(x_1-0.5)^2-10(x_2-0.5)^2}\dd x_1\dd x_2$}\vspace{0.25em} \\ \cline{1-5}\cline{1-5}
        $10^{-1}$ & $2.02*10^{-3}$ ($ 0.27$ sec.) & $5.17*10^{-3}$ ($  1.04$ sec.) & $1.88*10^{-2}$ ($  4.49$ sec.) & $1.09*10^{-2}$ ($  21.28$ sec.) \\\hline
        $10^{-2}$ & $5.00*10^{-7}$ ($ 3.39$ sec.) & $3.24*10^{-5}$ ($ 17.38$ sec.) & $1.28*10^{-3}$ ($ 59.73$ sec.) & $9.72*10^{-4}$ ($ 239.21$ sec.) \\\hline
        $10^{-3}$ & $2.08*10^{-8}$ ($39.50$ sec.) & $3.43*10^{-6}$ ($190.07$ sec.) & $4.50*10^{-4}$ ($656.52$ sec.) & $1.64*10^{-7}$ ($2527.79$ sec.) \\\hline
        \hline
    \end{tabular}}
    \caption{Error and runtime of the \eqref{prob3_ODE} solver for  2-d problems with $P = \icol{0.5\\0.5}$}\label{tab:prob3_2d}
\end{table}

\begin{figure}[ht]
    \centering
    \begin{subfigure}{.19\textwidth}
        \centering
        \includegraphics[width=\linewidth]{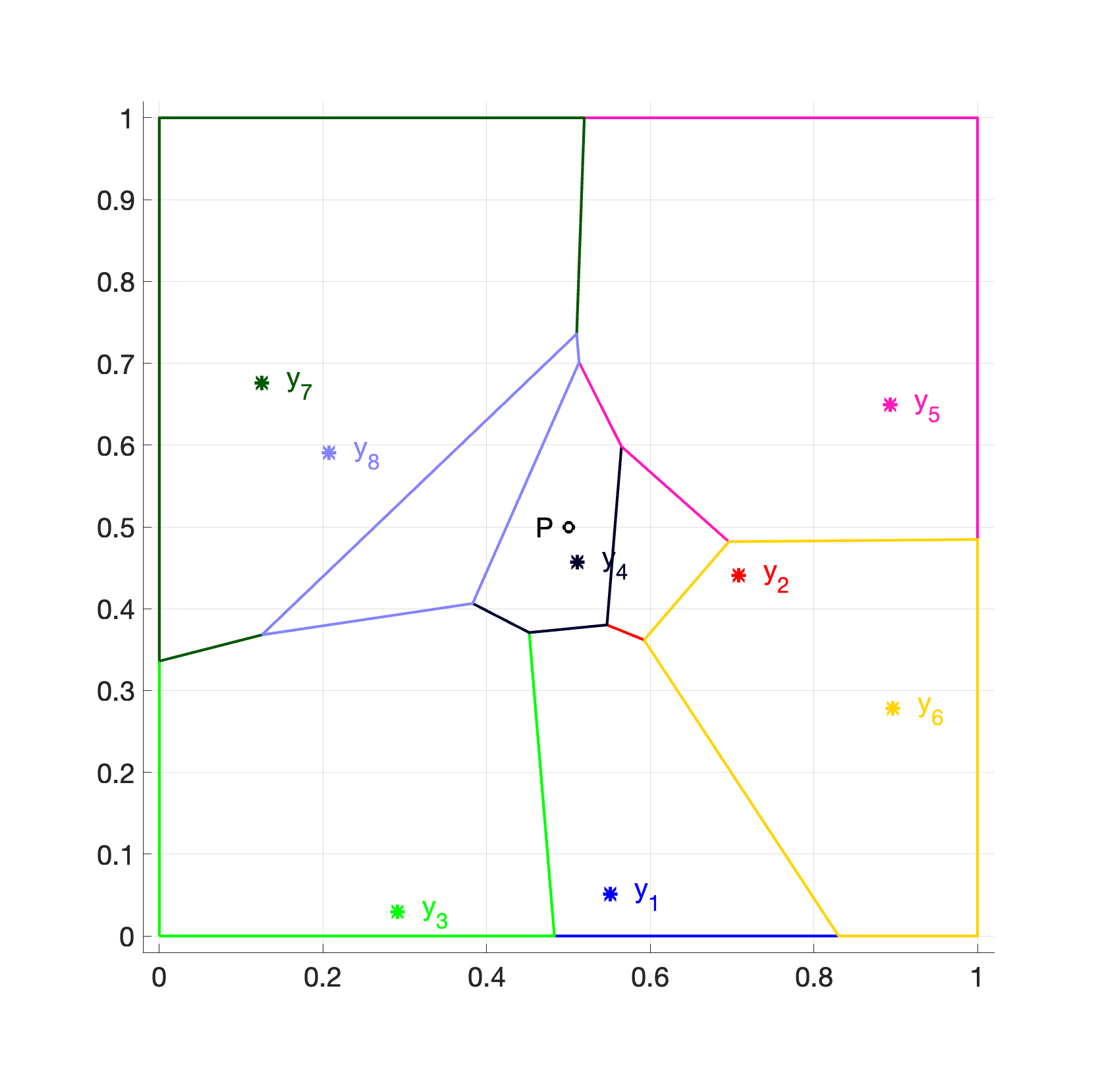}
        \caption{$t=0$}
    \end{subfigure}
    \begin{subfigure}{.19\textwidth}
        \centering
        \includegraphics[width=\linewidth]{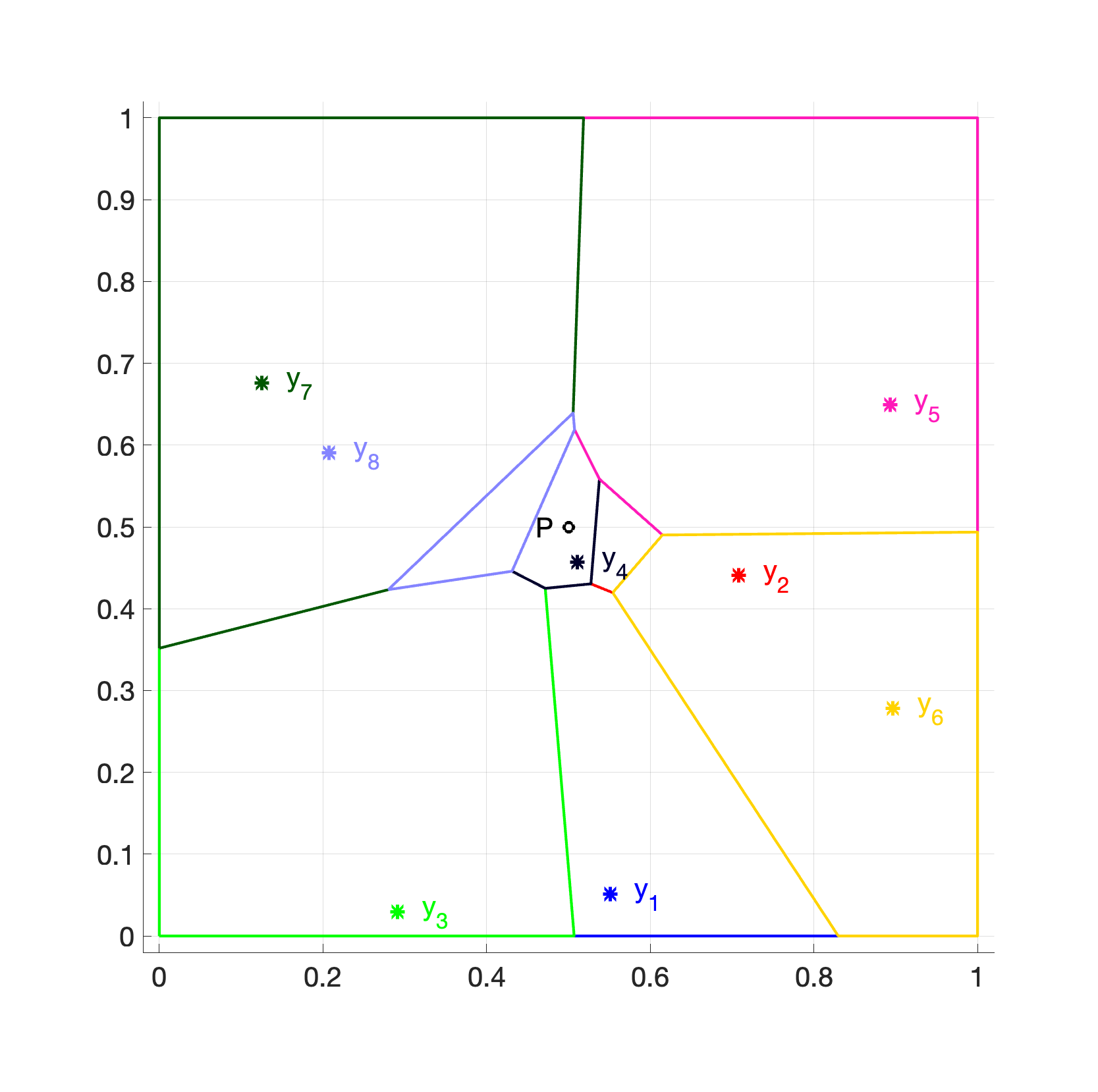}
        \caption{$t=0.25$}
    \end{subfigure}
    \begin{subfigure}{.19\textwidth}
        \centering
        \includegraphics[width=\linewidth]{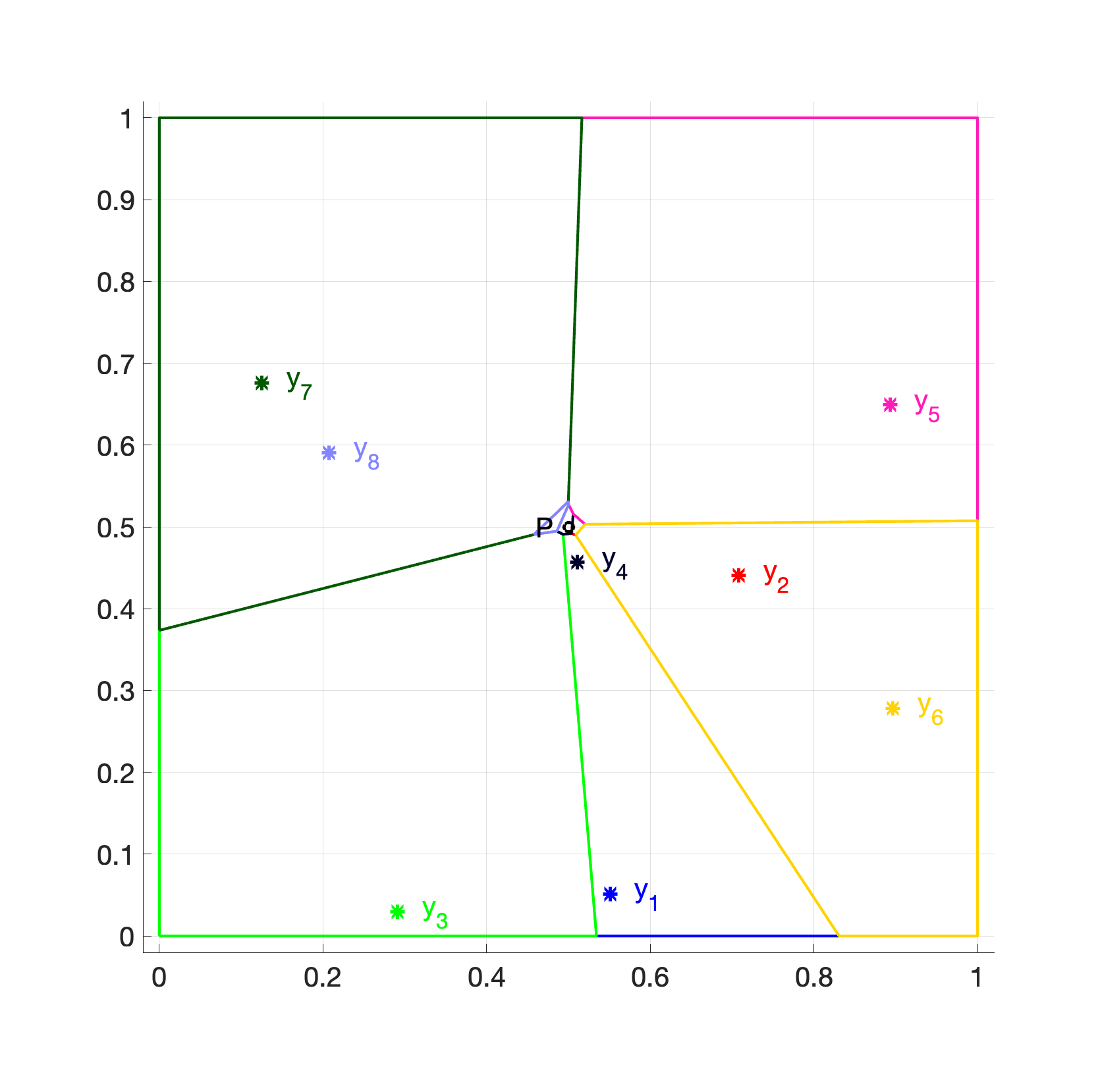}
        \caption{$t=0.5$}
    \end{subfigure}
    \begin{subfigure}{.19\textwidth}
        \centering
        \includegraphics[width=\linewidth]{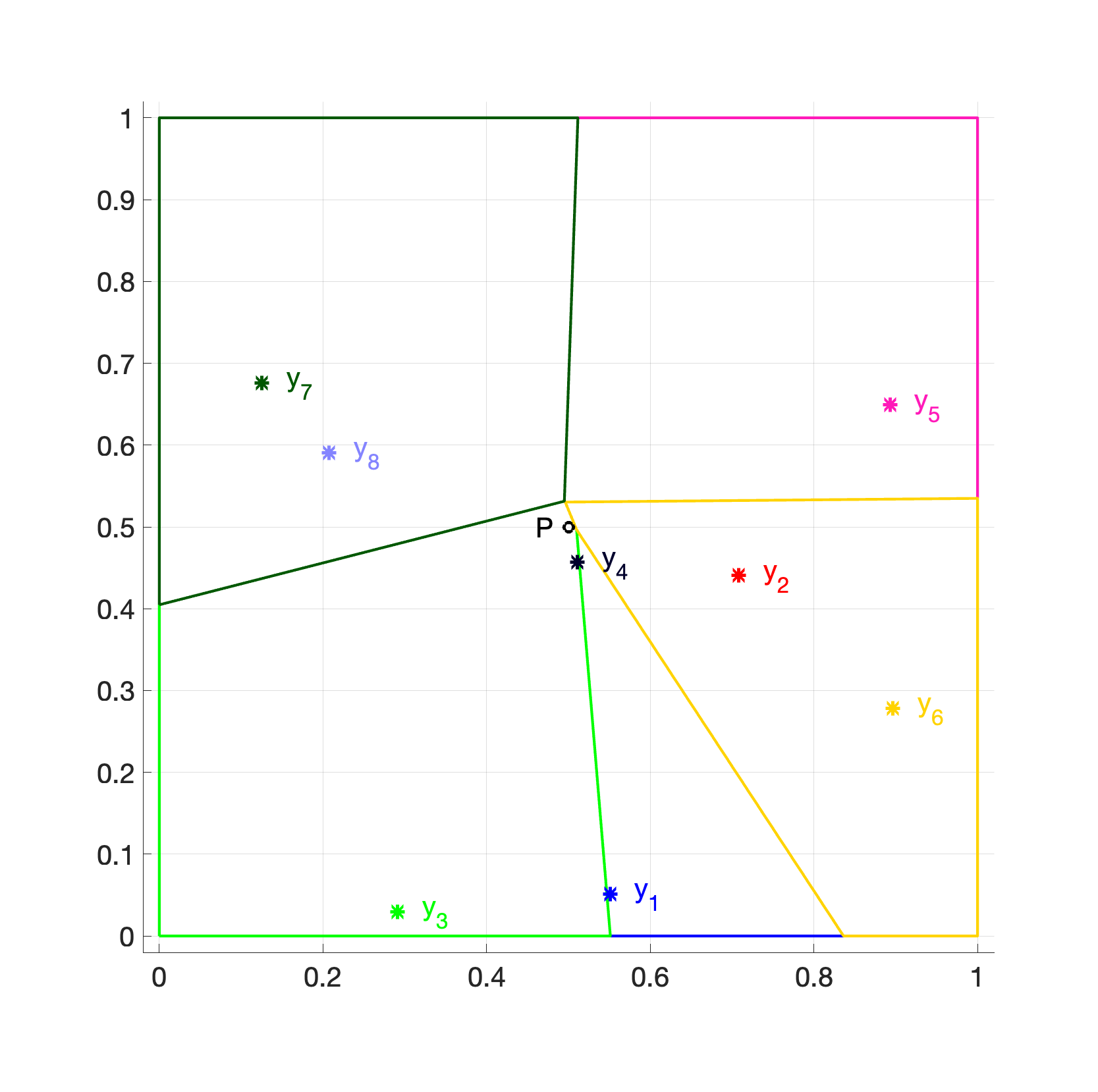}
        \caption{$t=0.75$}
    \end{subfigure}
    \begin{subfigure}{.19\textwidth}
        \centering
        \includegraphics[width=\linewidth]{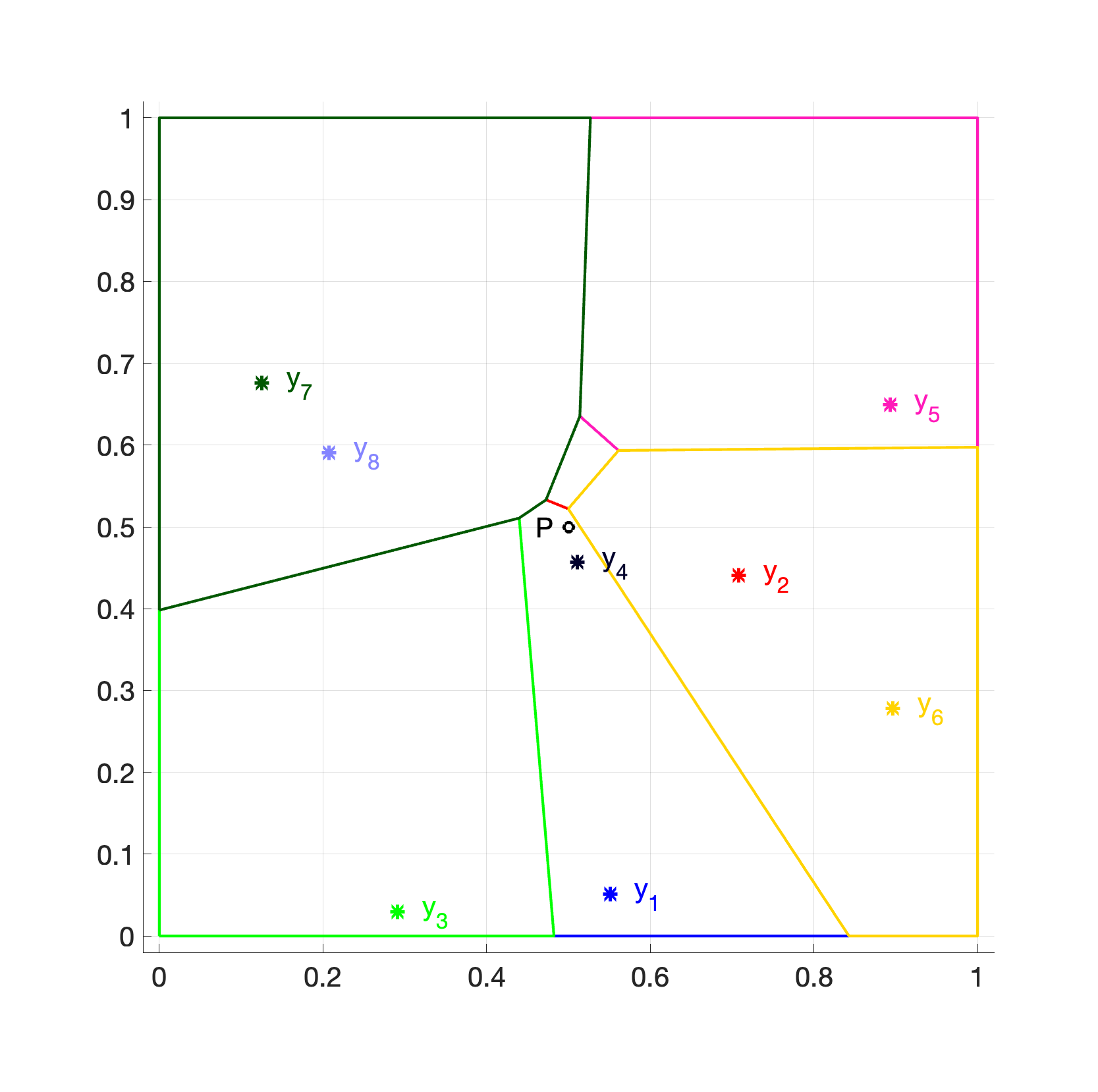}
        \caption{$t=1$}
    \end{subfigure}
    \caption{Time evolution of Laguerre cells with 8 random points and uniform measure}
    \label{fig:prob3_2d_unif}
\end{figure}

Table \ref{tab:prob4_2d} reports the error and runtime for the 2-d solver \eqref{prob4_ODE} with $\dd\mu(x) = \dd x$ and $\dd\rho(x) = 3.3508e^{-10(x_1-0.5)^2-10(x_2-0.5)^2}\dd x_1\dd x_2$. The error decreases as \(t\) is refined, showing empirical convergence close to first order for small \(N\) but deteriorating to sublinear for the larger \(N\) values; runtime grow steeply with both \(N\) and \(1/\Delta t\).

\begin{table}[ht]
    \centering\scalebox{0.75}{\setstretch{2.5}
    \begin{tabular}{||c||c|c|c|c||}\hline\hline
        $\Delta\mathbf{ t}$ & $N = 3$ & $N = 6$ & $N = 12$ & $N = 24$\\\hline\hline
        $10^{-1}$ & $1.96*10^{-2}$ ($  2.31$ sec.) & $1.52*10^{-2}$ ($ 17.06$ sec.) & $2.53*10^{-2}$ ($  47.37$ sec.) & $1.01*10^{-1}$ ($  96.02$ sec.) \\\hline
        $10^{-2}$ & $8.93*10^{-4}$ ($ 20.53$ sec.) & $5.92*10^{-4}$ ($ 78.82$ sec.) & $2.24*10^{-3}$ ($ 175.36$ sec.) & $5.04*10^{-2}$ ($ 467.69$ sec.) \\\hline
        $10^{-3}$ & $8.44*10^{-5}$ ($216.23$ sec.) & $5.23*10^{-5}$ ($725.76$ sec.) & $8.73*10^{-4}$ ($1481.89$ sec.) & $4.98*10^{-2}$ ($4290.78$ sec.) \\\hline
        \hline
    \end{tabular}}
    \caption{Error and time for \eqref{prob4_ODE} with $\dd\mu(x) = \dd x$ and\\ $\dd\rho(x) = 3.3508e^{-10(x_1-0.5)^2-10(x_2-0.5)^2}\dd x_1\dd x_2$}\label{tab:prob4_2d}
\end{table}

Finally, we solve the problem when target points are placed along a scaled parabola using the IVP-s \eqref{prob1_ODE} and \eqref{prob2_ODE}:
\begin{gather*}
    \text{Scaled parabola:} \quad y(t) = \icol{t \\ \left(\frac{t}{e}\right)^2}, \quad t \in [0, 1]\ .
\end{gather*}
For a given value of $N$, the target points are distributed uniformly with respect to the parameter $t$ over the interval $[0,1]$. In \cite{halim2025_SDM21D} it was shown that, in the unregularized case, this problem admits a nested structure; in terms of Laguerre cells, the triple intersections of successive cells are empty:
\[
    \mathrm{Lag}_{i-1}(\mathbf{v}) \cap \mathrm{Lag}_i(\mathbf{v}) \cap \mathrm{Lag}_{i+1}(\mathbf{v}) = \emptyset\ , \ i=2,3,\dots, N-1\ .
\]
Table \ref{tab:ScalPar_2d} reports the numerical results obtained with the ODE solvers and illustrates the second-order convergence in $t$. A comparison with Newton's method and with the sequential method (which directly exploits the nested structure), as reported in Table 1 of \cite{halim2025_SDM21D}, indicates an approximately order-of-magnitude disadvantage in computational time for the ODE-based approach. This behavior is to be expected: the initial-value formulation provides, alongside the terminal solution, a continuous evolution of the Laguerre cells, tracing the transition from the entropy-regularized state to the solution of the original problem (see Figure \ref{fig:Scal_Par}); such transient information is not produced by Newton's method or sequential algorithm.

\begin{table}[ht]
    \centering\scalebox{0.75}{\setstretch{2.5}
    \begin{tabular}{||c||c|c|c|c||}\hline\hline
        $\Delta\mathbf{ t}$ & $N = 3$ & $N = 6$ & $N = 12$ & $N = 24$\\\hline\hline
 
        \multicolumn{5}{c}{\Large IVP \eqref{prob1_ODE}} \vspace{0.25em}\\\cline{1-5}\cline{1-5}
        $10^{-1}$ & $1.24*10^{-2}$ ($ 0.90$ sec.) & $2.90*10^{-2}$ ($  0.65$ sec.) & $1.82*10^{-2}$ ($  2.90$ sec.) & $1.91*10^{-2}$ ($  13.96$ sec.) \\\hline
        $10^{-2}$ & $6.96*10^{-7}$ ($ 6.55$ sec.) & $4.95*10^{-4}$ ($ 10.45$ sec.) & $4.02*10^{-3}$ ($ 39.93$ sec.) & $7.37*10^{-3}$ ($ 196.37$ sec.) \\\hline
        $10^{-3}$ & $9.62*10^{-5}$ ($74.80$ sec.) & $2.45*10^{-6}$ ($144.11$ sec.) & $4.72*10^{-6}$ ($460.54$ sec.) & $6.84*10^{-4}$ ($2113.70$ sec.) \\\hline
        \hline

        \multicolumn{5}{c}{\Large IVP \eqref{prob2_ODE}} \vspace{0.25em}\\\cline{1-5}\cline{1-5}
        $10^{-1}$ & $1.40*10^{-2}$ ($ 0.62$ sec.) & $3.04*10^{-2}$ ($  1.58$ sec.) & $1.36*10^{-2}$ ($  6.51$ sec.) & $1.57*10^{-2}$ ($  31.07$ sec.) \\\hline
        $10^{-2}$ & $5.03*10^{-4}$ ($ 8.03$ sec.) & $1.60*10^{-3}$ ($ 21.46$ sec.) & $3.65*10^{-3}$ ($ 84.35$ sec.) & $6.43*10^{-3}$ ($ 411.69$ sec.) \\\hline
        $10^{-3}$ & $1.83*10^{-3}$ ($94.02$ sec.) & $2.96*10^{-5}$ ($246.72$ sec.) & $3.85*10^{-5}$ ($924.54$ sec.) & $5.63*10^{-4}$ ($4357.13$ sec.) \\\hline
        \hline
    \end{tabular}}
    \caption{Error and runtime for scaled parabola problem with uniform measure}\label{tab:ScalPar_2d}
\end{table}

\begin{figure}[ht]
    \centering
    \begin{subfigure}{.19\textwidth}
        \centering
        \includegraphics[width=\linewidth]{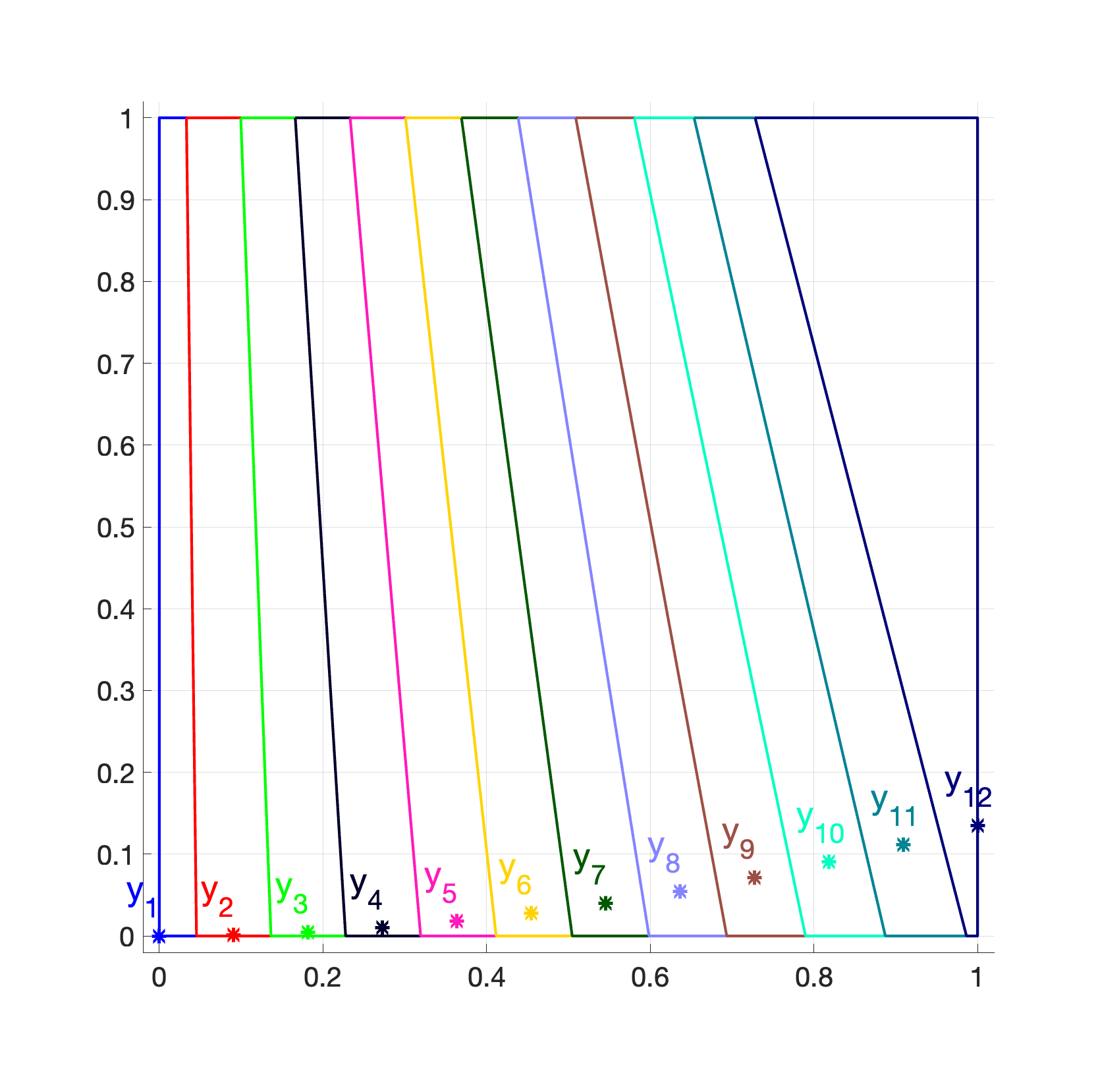}
        \caption{$t=0$}
    \end{subfigure}
    \begin{subfigure}{.19\textwidth}
        \centering
        \includegraphics[width=\linewidth]{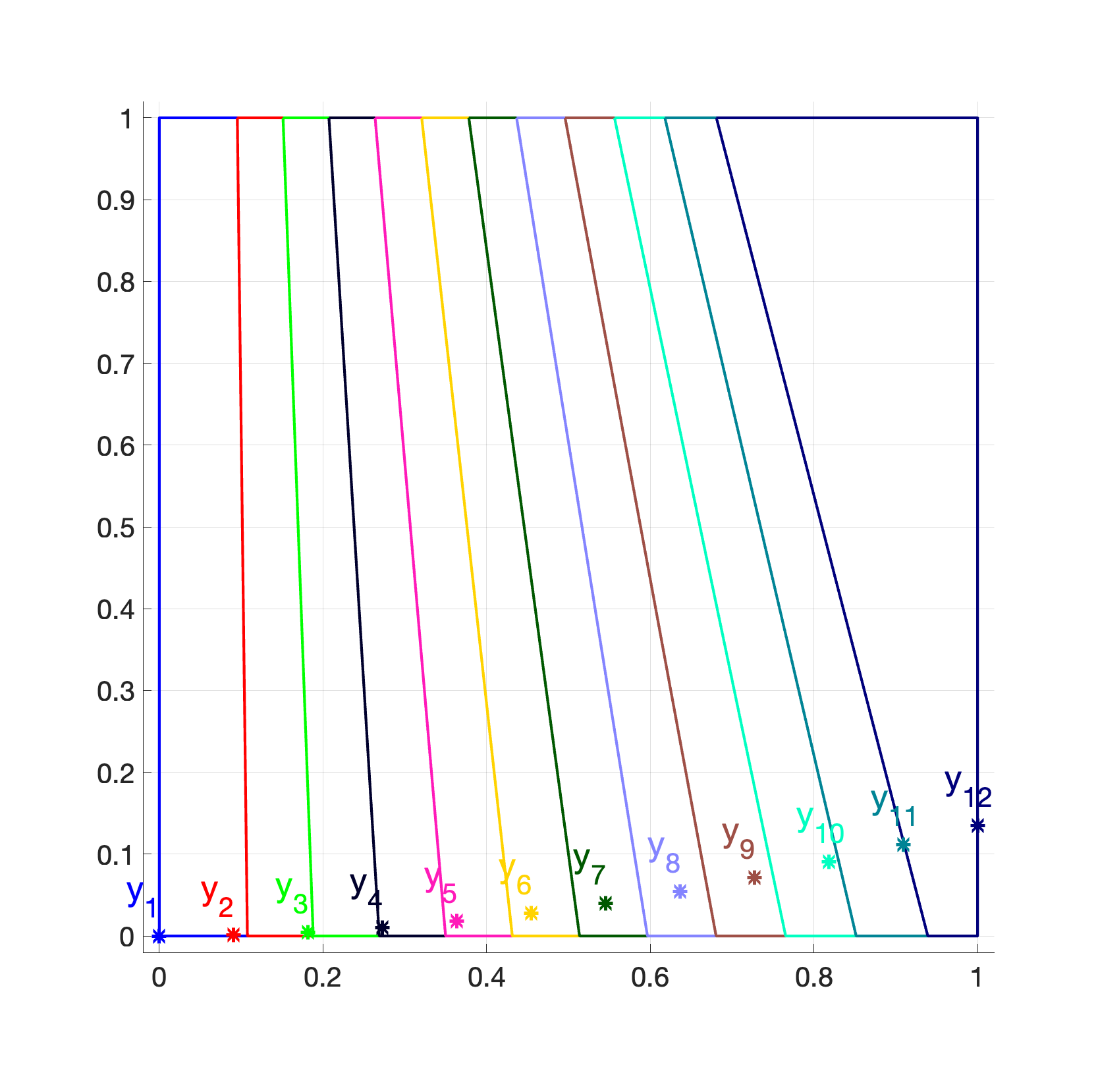}
        \caption{$t=0.25$}
    \end{subfigure}
    \begin{subfigure}{.19\textwidth}
        \centering
        \includegraphics[width=\linewidth]{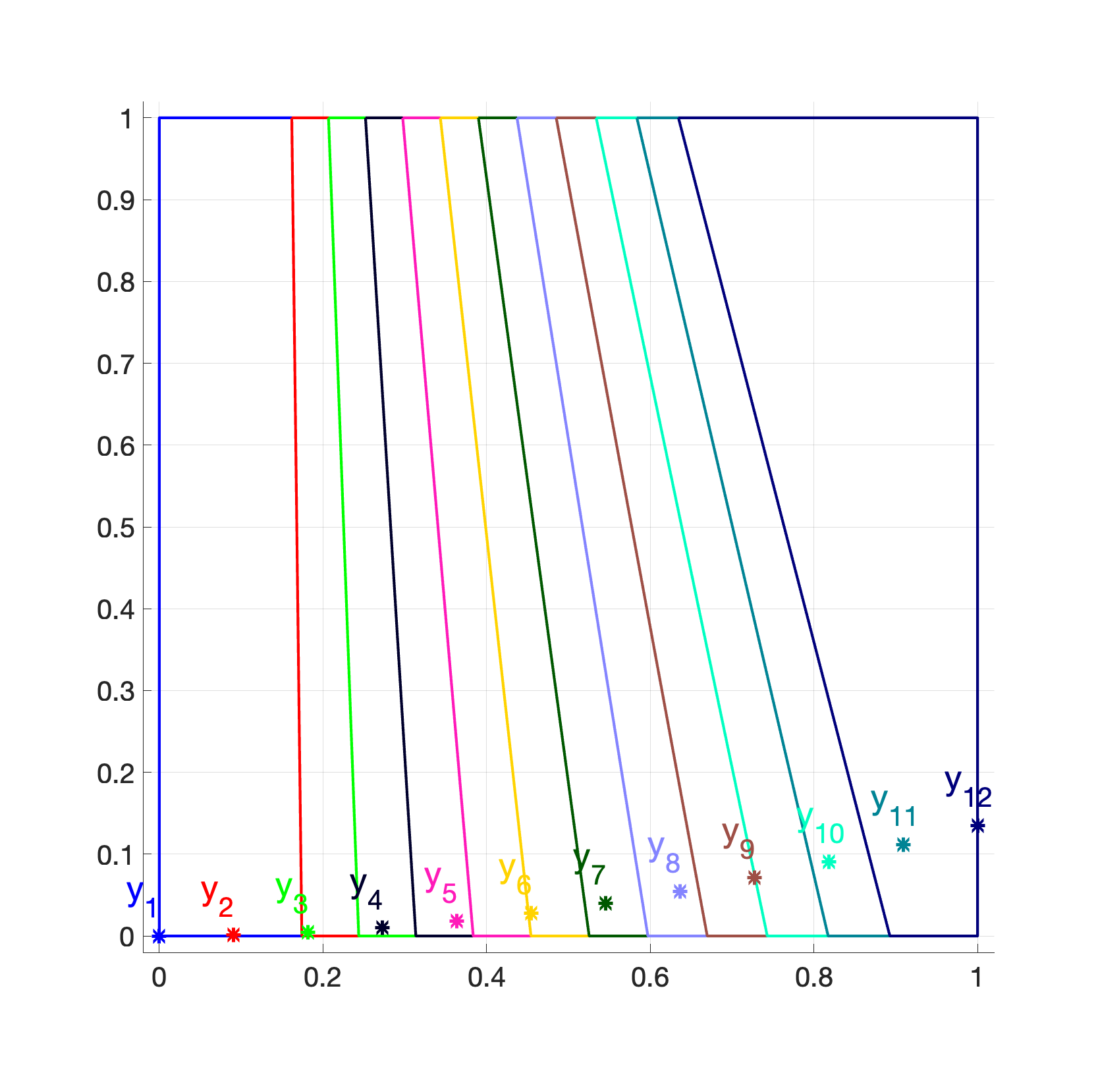}
        \caption{$t=0.5$}
    \end{subfigure}
    \begin{subfigure}{.19\textwidth}
        \centering
        \includegraphics[width=\linewidth]{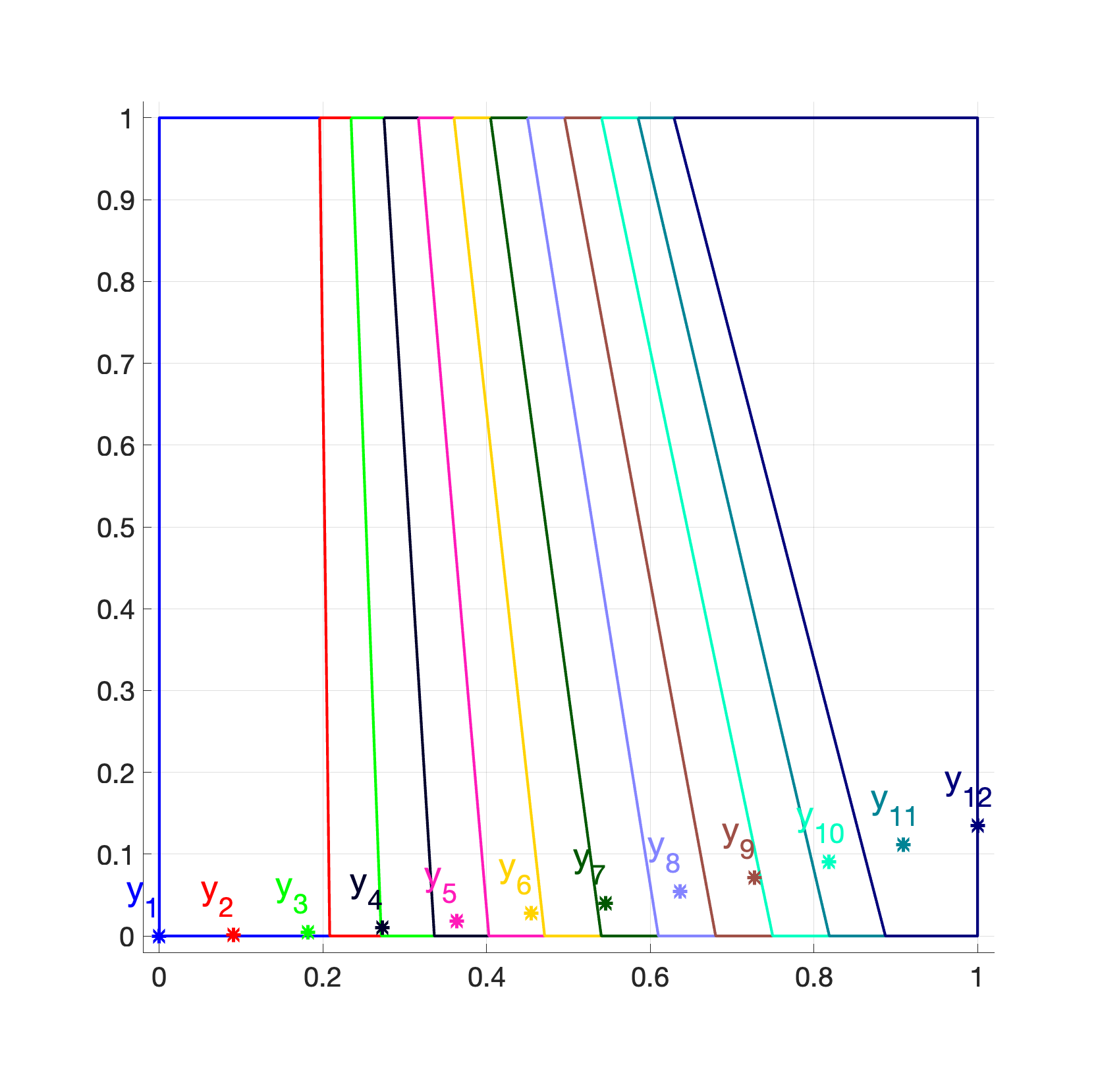}
        \caption{$t=0.75$}
    \end{subfigure}
    \begin{subfigure}{.19\textwidth}
        \centering
        \includegraphics[width=\linewidth]{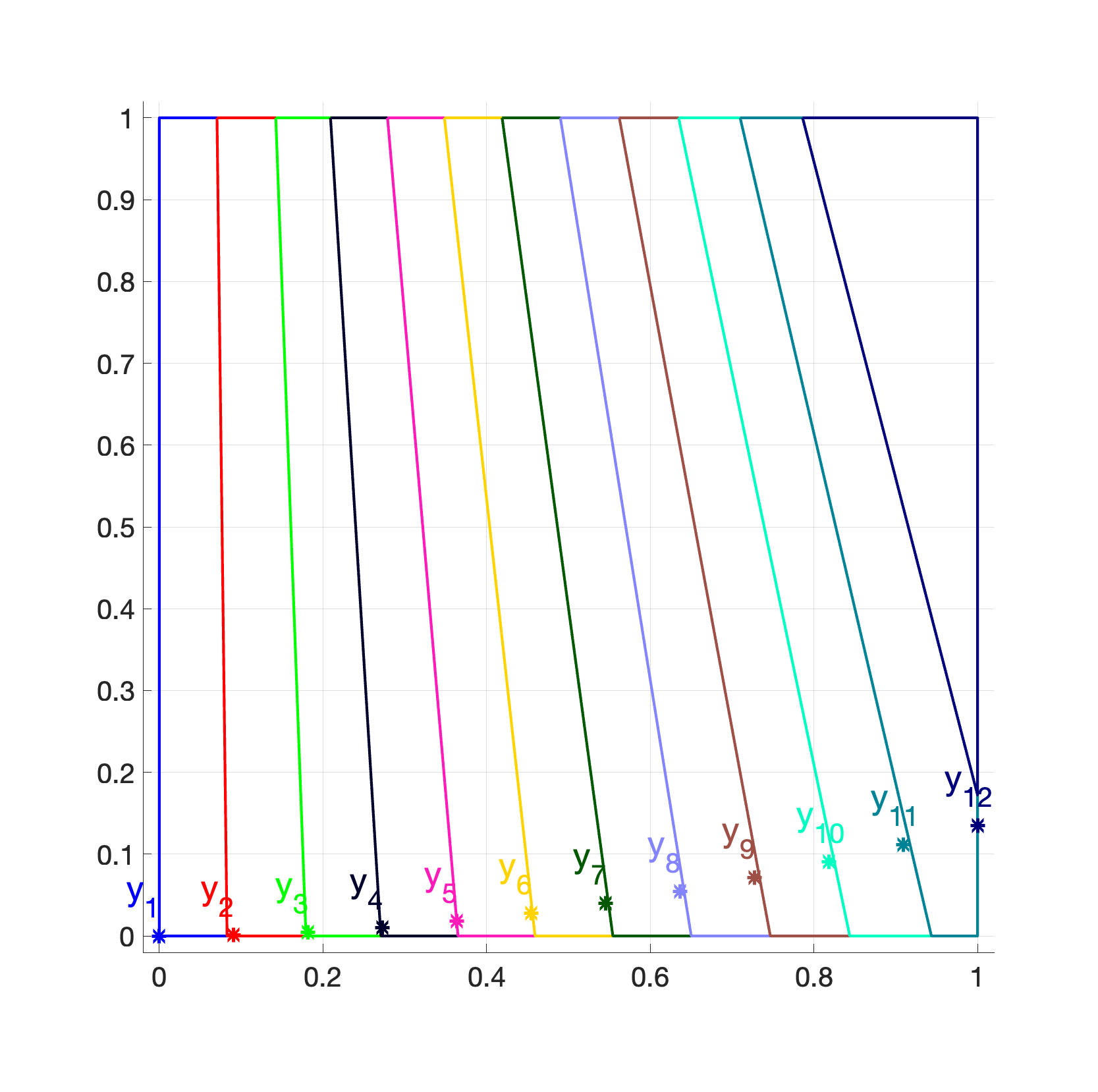}
        \caption{$t=1$}
    \end{subfigure}
    \caption{Time evolution of Laguerre cells with $12$ target points along the scaled parabola and a uniform measure, computed using the \eqref{prob1_ODE} solver}
    \label{fig:Scal_Par}
\end{figure}

\subsection*{Acknowledgments}
A.C. has been partially supported by the UDOPIA doctoral program, as well as the Agence nationale de la recherche, through the ANR project GOTA (ANR-23-CE46-0001) and the PEPR PDE-AI project (ANR-23-PEIA-0004).

\noindent L.N. benefited from the support of the FMJH Program PGMO and from the ANR project GOTA (ANR-23-CE46-0001). 
 
\noindent D.O. gratefully acknowledge that this research was supported in part by the Pacific Institute for the Mathematical Sciences.
 
\noindent B.P. is pleased to acknowledge the support of Natural Sciences and
Engineering Research Council of Canada Discovery Grant number  04864-2024.
\printbibliography

\end{document}